\documentclass[11pt]{amsart}

\usepackage{amsfonts, amstext, amsmath, amsthm, amscd, amssymb}
\usepackage{epsfig, graphics, overpic, enumerate}
\usepackage{color}
\let\oldmarginpar\marginpar
\renewcommand\marginpar[1]{\oldmarginpar[\raggedleft\footnotesize #1]%
{\raggedright\footnotesize #1}}

\renewcommand{\setminus}{{\smallsetminus}}

\newcommand{\RR}{{\mathbb{R}}}

\newcommand{\bdy}{{\partial}}
\newcommand{\guts}{{\rm guts}}
\newcommand{\cut}{{\backslash \backslash}}
\newcommand{\clock}{{\phi}}
\newcommand{\vol}{{\rm vol}}

\newcommand{\GA}{{\mathbb{G}_A}}
\newcommand{\GB}{{\mathbb{G}_B}}
\newcommand{\G}{{\mathbb{G}}}
\newcommand{\GRA}{{\mathbb{G}'_A}}
\newcommand{\GRB}{{\mathbb{G}'_B}}

\def\co{\colon\thinspace}

\theoremstyle{plain}
\newtheorem{theorem}{Theorem}[section]
\newtheorem{corollary}[theorem]{Corollary}
\newtheorem{lemma}[theorem]{Lemma}

\newtheorem{conjecture}[theorem]{Conjecture}

\newtheorem*{namedtheorem}{\theoremname}
\newcommand{\theoremname}{testing}

\theoremstyle{definition}
\newtheorem{define}[theorem]{Definition}
\newtheorem{remark}[theorem]{Remark}
\newtheorem{convention}[theorem]{Convention}

\newtheorem{example}[theorem]{Example}

\begin{document}
\title{Hyperbolic  semi-adequate links}
\author[D. Futer]{David Futer}
\author[E. Kalfagianni]{Efstratia Kalfagianni}
\author[J. Purcell]{Jessica S. Purcell}

\address[]{Department of Mathematics, Temple University,
Philadelphia, PA 19122, USA}

\email[]{dfuter@temple.edu}

\address[]{Department of Mathematics, Michigan State University, East
Lansing, MI 48824, USA}

\email[]{kalfagia@math.msu.edu}

\address[]{ Department of Mathematics, Brigham Young University,
Provo, UT 84602, USA}

\email[]{jpurcell@math.byu.edu }
\thanks{{D.F. is supported in part by NSF grants DMS--1007221 and DMS--1408682.}}

\thanks{{E.K. is supported in part by NSF grants DMS--1105843 and DMS--1404754.}}

\thanks{{J.P. is supported in part by NSF grants DMS--1007437, DMS--1252687, and a
    Sloan Research Fellowship.}}

\thanks{ \today}

\begin{abstract}
We provide a diagrammatic criterion for semi-adequate links to be hyperbolic. We also give a conjectural description of the satellite structures of semi-adequate links. One application of our result is that the closures of sufficiently complicated positive braids are hyperbolic links.
\end{abstract}

\maketitle

\section{Introduction}

The problem of determining the geometric structure of a link
complement from a link diagram  is both important and hard.
A related, similarly difficult, problem asks for relations between
geometric and diagrammatic invariants of a link.
The purpose of this paper is to discuss these problems for the class of
\emph{semi-adequate} links.  We give
diagrammatic criteria for such links to be hyperbolic, and state a
conjecture about their satellite structures.  Semi-adequate links form
a very broad class of links that first appeared in the study of
Jones--type invariants \cite{lick-thistle, thi:adequate}, and have since
been studied considerably from the point of view of both quantum topology
and geometric topology; see \cite{fkp:survey} and references therein.

In  \cite{fkp:gutsjp}, we
developed a framework for establishing relations between geometric and
combinatorial link invariants. In particular,  to a semi-adequate link diagram we associate a certain  graph (\emph{state graph})
and a surface spanned by the link, and construct a certain ideal polyhedral decomposition of the surface complement.
We use normal surface theory to show that combinatorial properties  of the state graph dictate the structure of the
JSJ-decomposition of the surface complement, and encode geometric information of the link complement.
For instance we show that,  for hyperbolic semi-adequate links,
graph theoretic invariants  coarsely determine the volume of the link
\cite{fkp:gutsjp} and the geometric types of certain essential surfaces in the link complement  \cite{fkp:qsf}.
The machinery of \cite{fkp:gutsjp} lends itself naturally
to the study of essential surfaces in link complements via
normal surface theory.

 In this paper, we focus on
essential tori and annuli in link complements,
and give a diagrammatic  criterion that rules them out, implying the link is hyperbolic.
For links that fail this criterion, we give a 
conjectural description of the satellite structures.
Our results place several known classes of hyperbolic links under a
common umbrella and lead to new constructions of such links.

To  state our results, we need to briefly explain the related  terminology; 
we give precise definitions in Section~\ref{sec:background}. For every 
semi-adequate link diagram  there is a corresponding state graph $\G$. 
The edges of $\G$ are in one-to-one correspondence with the crossings 
of the link diagram. One way to obtain 2--edge loops in $\G$ is from 
crossings in the same twist region of the link diagram, where edges of 
$\G$ corresponding to crossings of that twist region are parallel between 
two vertices of $\G$. Our result concerns link diagrams for which all the 
2--edge loops of $\G$ are obtained this way.

\begin{theorem}\label{thm:main}
Suppose that $D(K)$ is a connected, prime, semi-adequate diagram with 
at least two twist regions, such that for each 2--edge loop in the 
corresponding state graph, the edges belong to the same twist region. 
Then the link $K$ depicted by this diagram is hyperbolic.
\end{theorem}

Figure~\ref{fig:pretzel} shows an example of a connected, prime, semi-adequate diagram that doesn't satisfy
 the 2--edge loop hypothesis of Theorem \ref{thm:main}.
This diagram  represents
the $(-2,3,3)$ pretzel knot,  which is known to be equivalent to the $(3,4)$ torus knot, hence is not hyperbolic.
This shows that the  2--edge loop  condition is necessary for Theorem
 \ref{thm:main}. The class of links with semi-adequate diagrams that don't satisfy the 2--edge loop condition
is quite large, and contains plenty of hyperbolic knots (e.g.\ the $(-2,3,7)$ pretzel) and satellite knots
(see Example \ref{ex:satellite}).

\begin{figure}
\includegraphics{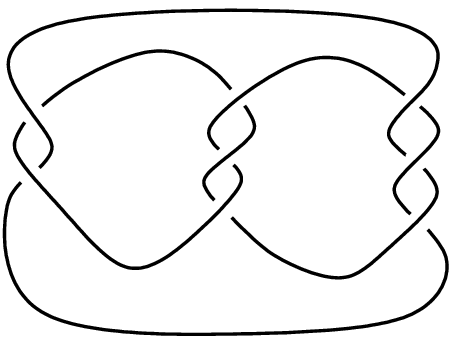}
\hspace{.2in}
\includegraphics{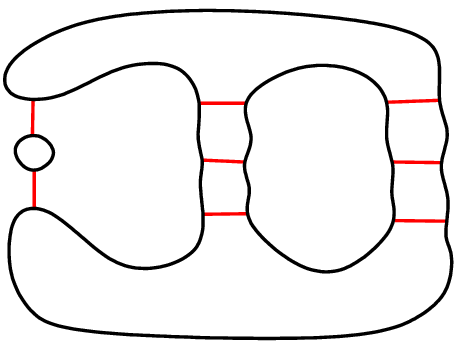}
\caption{Left: the $(-2,3,3)$ pretzel knot. Right: the graph of the all--$A$ resolution contains $2$--edge loops that do not belong to a single twist region.}
\label{fig:pretzel}
\end{figure}

Theorem \ref{thm:main} is reminiscent of a result of
Menasco, which states that any link admitting a connected, prime,
alternating diagram with at least two twist regions is hyperbolic
\cite{menasco:incompress}.  In fact, the hypotheses of Theorem
\ref{thm:main} apply in particular to prime, alternating diagrams.  In
this setting, the statement of the theorem reduces to Menasco's result.

In addition, Theorem \ref{thm:main} generalizes Menasco's result to
large classes of non-alternating links.  For a warm-up example,
consider the non-alternating pretzel link diagram of the form $P(a_1,
\dots, a_r, b_1, \ldots,b_s)$ that has $r$ vertical
bands containing $ a_1, \dots, a_r$ positive crossings, and $s$
vertical bands containing $b_1, \ldots,b_s$ negative crossings. See Figure~\ref{fig:pretzel} for the example of $P(-2,3,3)$, and
\cite[Figures 1 and 2]{lick-thistle} for the general case.  If $r, s\geq 3$ and $a_i,
b_i\geq 3$, for all $i$, then the diagrams satisfy the hypothesis of
Theorem \ref{thm:main}.  More generally, one may obtain families of
non-alternating Montesinos or arborescent  links by imposing similar restrictions on
the rational tangles involved. Note that a classification of hyperbolic arborescent links 
is already known
 by the work of Bonahon and Siebenmann
\cite{bonsieb:monograph, fg:arborescent}.

For a sample class of non-alternating links whose
hyperbolicity can be established for the first time using Theorem
\ref{thm:main}, consider the family of positive or negative closed
braids with at least $3$ crossings per twist region.  

\begin{corollary} \label{braids}
Let $B_n$ be the
braid group on $n$ strands, with $n \geq 3$, and let $\sigma_1,
\ldots, \sigma_{n-1}$ be the elementary braid generators. Let $b=\sigma_{i_1}^{r_1}\sigma_{i_2}^{r_2} \cdots \sigma_{i_k}^{r_k}$
be a braid in $B_n$.  
Suppose that either $r_j \geq 3$ for all $j$, or
else $r_j \leq -3$ for all $j$. Suppose moreover that the braid closure $D_b$ of $b$ is a
prime diagram.
Then the link $K$ depicted by this diagram is
hyperbolic.
\end{corollary}

Several other applications of Theorem \ref{thm:main} are
given by Giambrone \cite{giambrone}. For instance, he  proves that for a
sufficiently complicated braid $b$, the plat closure of $b$ is
hyperbolic.

The problem of determining the geometric structures of link
complements from link diagrams has been studied considerably in the
literature. In addition to the work of Menasco on alternating links,
Bonahon and Siebenmann \cite{bonsieb:monograph} classified the
geometric types of arborescent links and showed that with some
explicitly described exceptions, these links are hyperbolic.  See also
Futer and Gu\'eritaud \cite{fg:arborescent} for a direct proof.  Adams
showed that augmented alternating links are hyperbolic
\cite{adams:augmented}. He also showed that toroidally alternating links are
either composite, torus knots, or hyperbolic \cite{Adams-toroidally},
although determining which of the three occurs is difficult.   More recently,
Futer and Purcell showed that prime link diagrams in which each twist
region has at least six crossings represent either $(2, q)$ torus
links or hyperbolic links \cite{fp:twisted}. Purcell investigated
the geometric structures of certain families of links with multiply
twisted regions \cite{purcell2, purcell1}. For
similar results on other classes of knots and links, we refer the reader
 to Adams' survey paper \cite{Adams-survey}.

As a corollary of Theorem \ref{thm:main}, we conclude that the primality
of a link can be easily read off from a diagram.

\begin{corollary}\label{cor:prime}
Let $D(K)$ be a connected, semi-adequate diagram without nugatory
crossings.  Suppose that for each 2--edge loop in the corresponding
state graph, the edges belong to the same twist region of $D(K)$.
Then $K$ is a prime link if and only if $D(K)$ is prime.
\end{corollary}

Corollary \ref{cor:prime} is reminiscent of some prior results for more restricted link families. For instance, it generalizes a theorem of Menasco: if $D(K)$ is a connected
alternating diagram, then $K$ is prime if and only if  $D(K)$ is prime \cite{menasco:incompress}. Similarly, Ozawa proved that if $D(K)$ is a connected
positive diagram, then $K$ is prime if and only if  $D(K)$ is prime \cite{ozawa:positive}.

 In fact, the hypothesis
on $2$--edge loops in Corollary \ref{cor:prime} should be unnecessary.
It is conjectured that
 a connected, prime, semi-adequate
diagram must always represent a prime link (see 
\cite[Problem 10.6]{fkp:gutsjp} and \cite{ozawa}).  Corollary \ref{cor:prime} gives
a partial solution to this conjecture.

Finally, we note that a connected, semi-adequate diagram always represents a non-split link. See Thistlethwaite
\cite[Corollary 3.2]{thi:adequate} for the original proof, relying on properties of link polynomial invariants, and Ozawa \cite[Theorem 2.15]{ozawa} for an alternate, geometric proof.  
Thus for the rest of the paper,
we will assume our diagrams are connected, and hence the link is
non-split.

\subsection{Volume bounds and relations}

The machinery of \cite{fkp:gutsjp} allows for connections between
geometric invariants of a link complement, combinatorial properties of
its diagram, and stable coefficients of its colored Jones polynomials.
The class of links of Theorem \ref{thm:main} is particularly well suited for 
such applications.  For instance, 
\cite[Corollary 9.4]{fkp:gutsjp} relates the hyperbolic volume of these links to the
Euler characteristic of the corresponding reduced state graph. (See Definition \ref{def:adequate} for the terminology and notation.)

\begin{corollary}
Suppose that $K$ is a link with prime, semi-adequate diagram $D(K)$ as
in the statement of Theorem \ref{thm:main}.  Then
$$ \vol(S^3 \setminus K) \: \geq \: - v_8 \, \chi(\G'),
$$ where $v_8 = 3.6638...$ is the volume of a regular ideal octahedron
and $\G'$ is $\GRA$ or $\GRB$ according to whether $D(K)$ is
$A$--adequate or $B$--adequate.
\end{corollary}

This corollary has applications in two directions. 
First, coupled with the upper volume bounds given by Lackenby,  Agol, and Thurston,
\cite{lackenby:volume-alt}, and combined with additional work of Giambrone
\cite{giambrone},  it leads to two-sided bounds on $\vol(S^3
\setminus K)$ in terms of the twist number of a semi-adequate diagram.
This extends a result of Lackenby \cite{lackenby:volume-alt} and
results of the authors \cite{fkp:filling, fkp:conway, fkp:farey} to
new link families.  Second, it leads to two-sided bounds on $\vol(S^3 \setminus K)$ in
terms of stable coefficients of the colored Jones polynomials of $K$,
as predicted by the Coarse Volume Conjecture \cite[Question
  10.13]{fkp:gutsjp}.  Details are given in Giambrone \cite{giambrone}.

\subsection{Satellite semi-adequate links}

As we mentioned above, there is a conjectural strengthening of 
Corollary \ref{cor:prime}, which removes the hypothesis on 2--edge
loops.  Formulating the right
conjectural strengthening of Theorem \ref{thm:main} requires some care,
as there are many non-hyperbolic semi-adequate links.
For example, it is well-known that all torus links are
semi-adequate. Similarly, all planar cables of a semi-adequate diagram
are semi-adequate \cite{cromwell-book}. There are also many semi-adequate satellites, as
the following construction illustrates.

\begin{example}\label{ex:satellite}
Recall that a satellite link is constructed from a \emph{companion
  knot} $J \subset S^3$, a \emph{pattern link} $K' \subset D^2 \times
S^1$, and an embedding $f: D^2 \times S^1 \to N(J)$. The image $K =
f(K')$ will be a non-trivial satellite whenever $J$ is non-trivial, and $K'$ is
not the core of the solid torus or contained in a ball in the solid torus. The whole construction can
be performed diagrammatically: given a diagram $D(J) \subset \RR^2$,
and a diagram $D(K') \subset [0,1] \times S^1$, the blackboard framing
of $D(J)$ specifies a way to immerse the annulus $I \times S^1$ into
$\RR^2$, which gives a diagram $D(K)$. See Figure \ref{fig:semiad-satellite}.

Suppose $D(J) \subset \RR^2$ and $D(K') \subset I \times S^1$ are both
$A$--adequate diagrams. Consider the graph $H_A(K')$ coming from $K'$
(see Definitions \ref{def:HA} and \ref{def:adequate}). Suppose that
there is a rectangle $R = I \times I \subset I \times S^1$, such that
if we remove the state circles and segments of $H_A$ that lie entirely
in $R$, what remains is $n \geq 1$ copies of the core curve $\{\ast\}
\times S^1$.

Now, suppose that we use the blackboard framing of $D(J)$ to immerse
the annulus $I \times S^1$ into $\RR^2$, so that the image rectangle
$f(R)$ lies in a crossing--free region of $D(J)$. Outside the image
rectangle $f(R)$, the graph $H_A$ of the resulting diagram $D(K)$ will
look identical to the graph of the $n$--fold planar cable of $J$, which is
known to be $A$--adequate. Inside $f(R)$, the diagram is $A$--adequate
because $D(K')$ is $A$--adequate. See Figure \ref{fig:semiad-satellite} for
an example, which produces the Whitehead double of the trefoil.
\end{example}

\begin{figure}
\begin{overpic}{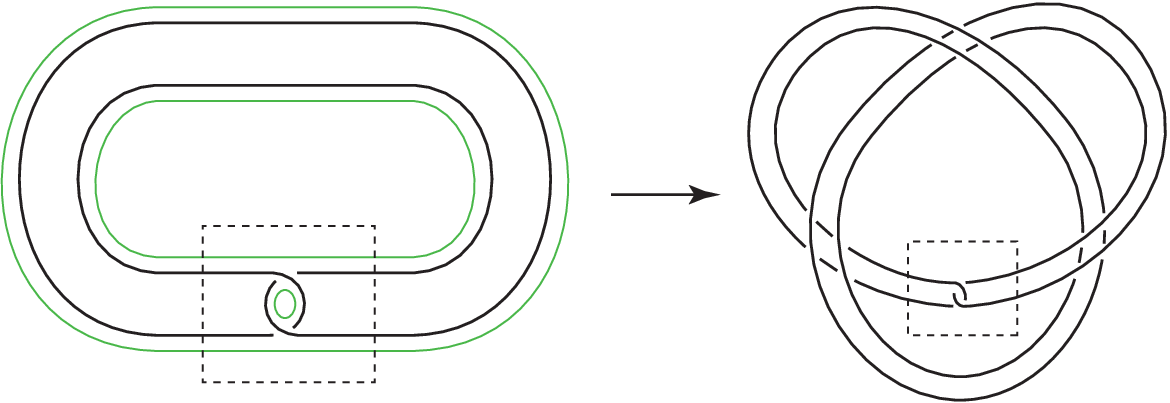}
\put(24,17){$R$}
\put(55,20){$f$}
\put(79,16){$f(R)$}
\end{overpic}
  \caption{Left: a diagram $D(K')$ in an annulus. State circles of the
    all--$A$ resolution are shown in green. After removing state
    circles in a rectangle $R$, what remains is two parallel copies of
    the core. Right: an embedding of this annulus into a regular
    neighborhood of the trefoil produces an $A$--adequate diagram of
    a Whitehead double.}
  \label{fig:semiad-satellite}
\end{figure}
 
\begin{conjecture}\label{conj:satellite}
Suppose $D(K)$ is a semi-adequate diagram of a satellite link $K$. Then $D(K)$ or its mirror 
image can be obtained using the construction of Example \ref{ex:satellite}. In particular, the satellite
torus must be visible in $D(K)$ as the regular neighborhood of an immersed annulus.
\end{conjecture}

\subsection{Organization}
The paper is organized as follows. In Section \ref{sec:background}, we
summarize definitions and the main tools from \cite{fkp:gutsjp} that
are needed in this paper.  To a semi-adequate diagram $D(K)$, we
associate a state surface $S_A$ that is essential in the complement of
$K$.  Its complement $M_A = S^3 \cut S_A$ admits a checkerboard ideal
polyhedral decomposition (see \S\ref{subsec:polyhedra}).  The
intersection of an essential torus in $S^3\setminus K$ with $M_A$ is a
collection of essential annuli that can be studied using normal
surface theory with respect to the polyhedral decomposition.  We
discuss this in Section \ref{sec:annuli}, where we also show that the
resulting essential annuli in $M_A$ fall into two types:
\emph{diagrammatically compressible} and \emph{diagrammatically
  incompressible} (see Definition \ref{def:diag-comp}).

In Sections \ref{sec:incompr} and \ref{sec:compr}, we analyze the two
types of annuli and conclude that under the hypotheses of Theorem
\ref{thm:main}, neither of the two types arise as part of an essential
torus.  We note that our analysis of diagrammatically incompressible
annuli in Section \ref{sec:incompr} does not require the hypotheses of
Theorem \ref{thm:main}.  It leads to a classification of such annuli
in complements of all semi-adequate links. We expect that these results
will have further applications, including in approaching Conjecture
\ref{conj:satellite}.  In Section \ref{sec:compr}, where we study
diagrammatically compressible annuli, a key ingredient is a
classification of \emph{essential product disks} in polyhedra from
\cite{fkp:gutsjp}.  The main result in Sections \ref{sec:incompr} and
\ref{sec:compr} is Theorem \ref{thm:tori}:  if a diagram $D(K)$ is as in
the statement of Theorem \ref{thm:main}, then $S^3\setminus K$ is
atoroidal.

In Section \ref{sec:seifert},  we complete the proof of Theorem \ref{thm:main}
and Corollaries \ref{braids} and \ref{cor:prime}.
Our approach in this section is to rule out the possibility
that $S^3 \setminus K$ might be Seifert fibered, using Gromov norm
estimates and Turaev surface methods.

%%%%%%%%%%%%%%%%%%%%%%%%%%%%%%%%%%%%%%%%%%%%%%%%%%%%%%%%%%%%%%%%%
%%%%%%%%%%%%%%%%%%%%%%%%%%%%%%%%%%%%%%%%%%%%%%%%%%%%%%%%%%%%%%%%%
\section{Background and tools}\label{sec:background}

We begin this section by recalling definitions of semi-adequate knots
and related terms.  We also review some constructions from
\cite{fkp:gutsjp} that we will be using throughout the paper.

\subsection{Definitions}

For any crossing of a link diagram $D:=D(K)$, there are two resolutions,
called the \emph{$A$--resolution} and \emph{$B$--resolution} of the
crossing, shown in Figure \ref{fig:splicing}.

\begin{figure}[h]
	\centerline{\input{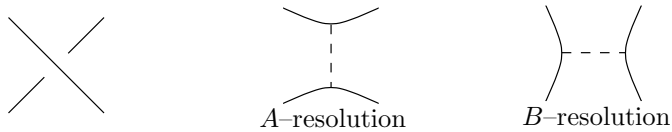}}
\caption{$A$-- and $B$--resolutions of a crossing.}
\label{fig:splicing}
\end{figure}

A choice of $A$-- or $B$--resolution for each crossing of $D$ is
called a \emph{Kauffman state} \cite{KaufJones}.  The result of applying a state to $D$
a collection of
circles disjointly embedded in the projection plane. These are called
\emph{state circles}. 

Throughout this paper, we will be concerned only with the all--$A$ and
all--$B$ states, which correspond to making a uniform choice of $A$ or $B$ at all crossings.

\begin{define}\label{def:HA}
Given a diagram $D(K)$ and the all--$A$ state of $D$, we construct a trivalent graph
$H_A$ as follows.  For each resolved crossing of $D$,  add an edge 
 between resulting state circles,
shown dashed in Figure \ref{fig:splicing}.  Every edge of $H_A$ either belongs to a
state circle of the all--$A$ resolution, or comes from a
crossing.  The latter edges are called \emph{segments}.

Similarly, we define a trivalent graph $H_B$, whose edges consist of
state circles and segments of the $B$--resolution. Observe that the original 
link diagram $D(K)$ can be reconstructed from the graph $H_A$ or $H_B$.
\end{define}

\begin{define}\label{def:adequate}
The \emph{state graphs} $\GA$ and $\GB$ are obtained from $H_A$ and
$H_B$, respectively, by collapsing each state circle to a vertex.
Removing redundant edges between vertices, we obtain the \emph{reduced
  state graphs} $\GRA$ and $\GRB$.

Following Lickorish and Thistlethwaite \cite{lick-thistle,
  thi:adequate}, a diagram $D$ is said to be \emph{$A$--adequate} if
every edge of $\GA$ has its endpoints on distinct vertices.
Similarly, one can define $B$--adequate diagrams using $\GB$. A link
diagram that is either $A$--adequate or $B$--adequate is called
\emph{semi-adequate}.
  
A link $K$ will be called $A$--adequate ($B$--adequate) if it admits
an $A$--adequate ($B$--adequate) diagram.  A link that is either
$A$--adequate or $B$--adequate is called \emph{semi-adequate}.
\end{define}

\begin{define}\label{def:twist}

A link diagram $D$ is called \emph{prime} if any closed curve in the
projection plane that meets the diagram transversely exactly twice
bounds a region of the diagram with no crossings.

A \emph{twist region} of $D$ is a collection of bigons in $D$ that are
adjacent end to end, such that there are no additional adjacent bigons on either end.  An example of such a twist region is shown at the top of Figure~\ref{fig:twist-resolutions}. A single crossing adjacent to no bigons is also a twist region.  
We require twist regions to be
alternating, for if $D$ contains a bigon that is not alternating, then
a Reidemeister move removes both crossings without altering the rest
of the diagram.
The number of distinct twist regions in a diagram is defined to be the
\emph{twist number} of that diagram. Note that if $D$ has exactly one twist region, it is a closed
2--braid; i.e.\ the standard diagram of a $(2, q)$ torus link.
\end{define}

To understand the statement of Theorem
\ref{thm:main}, we need to explain the hypothesis that for each
2--edge loop in the state graph, the edges belong to the same twist
region of $D(K)$.  To make this precise, we need the following
definition.

\begin{define}\label{def:long-short}
Suppose $R$ is a twist region of a link diagram $D$ such that $R$ contains $c_R>1$ crossings.  Consider the all--$A$ and all--$B$ resolutions applied to $R$.  One of the state graphs, say $\GB$, will inherit $c_R -1$ vertices from the $c_R -1$ bigons contained in $R$. We say that this is the \emph{long resolution} of $R$.  The other graph, say $\GA$, contains $c_R$ parallel edges only one of which survives in $\GRA$.  This is the \emph{short resolution} of $R$. See Figure~\ref{fig:twist-resolutions}.
 \end{define}

\begin{figure}[t]
\begin{center}
  \input{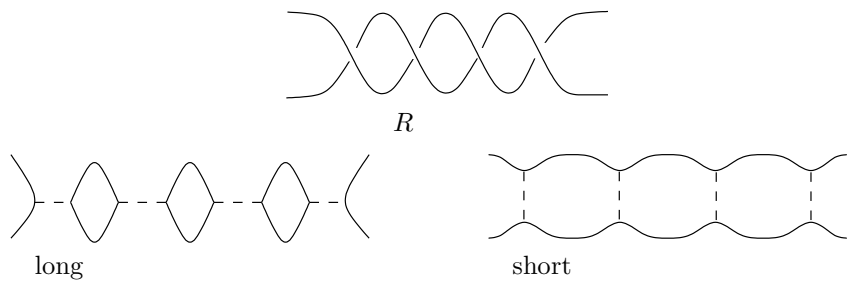}
\end{center}
\caption{Resolutions of a twist region $R$.}
\label{fig:twist-resolutions}
\end{figure}

Throughout the paper, we will be concerned with semi-adequate diagrams
where the 2--edge loops in the corresponding state graph come from
short resolutions of twist regions.

\begin{define}\label{def:2edges-belong}
Suppose that the state graph, say $\GA$, contains a 2--edge loop. We say that
the two edges of that loop \emph{belong to the same twist region} $R$
of the diagram if first, the edges come from resolving two crossings
in $R$, and second, the resolution of $R$ in $\GA$ is the short one.
If every $2$--edge loop in $\GA$ has its edges in the same twist region, 
we say that $\GA$ satisfies the \emph{2--edge loop
  condition}.  
\end{define}

Suppose $\gamma$ is a simple closed curve meeting the diagram $D(K)$ exactly
twice in two crossings $x_1, x_2$.  Adjust $\gamma$ in a neighborhood of each
crossing, so that after the adjustment $\gamma$ meets the diagram
exactly four times, and has a subarc $\gamma_i$ in the neighborhood of $x_i$, 
with endpoints on the projection of $K$.
Now consider the $A$-- and $B$--resolutions of the two crossings.  For
each $x_i$, exactly one of these resolutions will produce a segment
that is parallel to $\gamma_i$.  When $\gamma$ meets two crossings in
a twist region, then the resolution producing the segment parallel to
$\gamma_i$ is the short resolution of the twist region.

\begin{figure}[h]
  \input{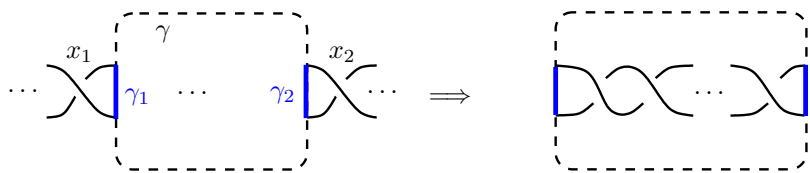}
  \caption{The property of being $A$--twist reduced. Whenever there is
    a closed curve $\gamma$ meeting the diagram as shown on the left,
    the crossings $x_1$ and $x_2$ must belong to the same twist
    region, as shown on the right. Note that the twist region
    containing $x_1, x_2$ can lie on either side of $\gamma$.}
  \label{fig:twist-reduced}
\end{figure}

\begin{define}\label{def:A-twist-reduced}
We say that a diagram is \emph{$A$--twist reduced} if it satisfies the
following property.  Suppose $\gamma$ is a simple closed curve
$\gamma$ meeting the diagram exactly four times adjacent to two
crossings, as above, such that the all--$A$ state produces segments
parallel to subarcs of $\gamma$.  Then $\gamma$ bounds a subdiagram
consisting of a (possibly empty) collection of bigons arranged in a
row between the two crossings. See Figure \ref{fig:twist-reduced}.
\end{define}

The property of being \emph{$B$--twist reduced} is defined in the same
way, with the $B$--resolution replacing the $A$--resolution. Adequate
diagrams that are both $A$-- and $B$--twist reduced are \emph{twist
  reduced} in the sense of Lackenby \cite{lackenby:volume-alt}.

In the sequel, we will consider $A$--adequate diagrams for
simplicity. The same results hold for $B$--adequate diagrams, by
taking a mirror image.

\begin{lemma}\label{lemma:twist-reduced}
Let $D(K)$ be an $A$--adequate diagram such that $\GA$ satisfies the
2--edge loop condition. Then $D(K)$ is $A$--twist reduced.
\end{lemma}

\begin{proof}
Suppose that $\gamma$ is a closed curve as in Definition
\ref{def:A-twist-reduced}, meeting the diagram exactly four times
adjacent to two crossings, $x_1$ and $x_2$.  Let $\gamma_i$ denote the
subarc of $\gamma$ which runs between two points on the diagram and
lies in a neighborhood of the crossing $x_i$.
 Suppose that the $A$--resolution of $x_1$ and $x_2$
produces segments parallel to $\gamma_1$ and $\gamma_2$.  Consider how
the state circles of the all--$A$ state intersect the region inside
$\gamma$.

First, note that if some state circle runs from $x_i$ back to $x_i$, this
state circle will violate the definition of $A$--adequacy.   Therefore, there must be
two state circles running from $x_1$ to $x_2$, and the two edges of $\GA$
corresponding to those crossings give a $2$--edge loop.  By
hypothesis, the edges belong to the same twist region $R$, which means
the edges come from resolving two crossings of the twist region $R$ in
$\GA$.  Then by definition there must be a (possibly
empty) collection of bigons between $x_1$ and $x_2$, as desired.
\end{proof}

%%%%%%%%%%%%%%%%%%%%%%%%%%%%%%%%%%%%%%%%%%%%%%%%%%%%%%%%%%%%%%%%%
\subsection{Polyhedral toolbox}\label{subsec:polyhedra}
The main technical tool that we use is a decomposition of the link complement into an $I$--bundle over a surface and a collection of ideal polyhedra \cite[Chapters 2--4]{fkp:gutsjp}.  In order to make this paper as self-contained as possible, we will review the definitions and constructions from there that are relevant to this paper.  We will also recall the statements of some key results that are needed below.  In the few occasions when we rely on results from \cite{fkp:gutsjp} that are not restated in this section, we will refer the reader to the exact statement of the result we are using in that monograph.  The background we provide here should suffice for checking and absorbing the statements in these cases. The reader need only consult the monograph \cite{fkp:gutsjp} in order to learn the detailed proofs.

For a reader who is new to this material, we also recommend consulting the survey paper \cite{fkp:survey} for a quick guide to the key features of the polyhedral decomposition.

\begin{define}\label{def:statesurface}
A diagram $D(K)$ determines a \emph{state surface} $S_A$, constructed as follows.  Each state circle of $H_A$ bounds a disk, and the disks associated to all the state circles can be disjointly embedded in the $3$--ball below the projection plane.  (Note this collection of disks in the lower $3$--ball is unique up to isotopy.)
Every crossing of $D(K)$ gives a segment of $H_A$, which runs between two state circles.  We connect the corresponding disks by a half--twisted band, twisted in the direction of the original crossing.  The result is a (possibly non-orientable) surface $S_A$, whose boundary is $K$.
\end{define}

When $D(K)$ is an $A$--adequate diagram, Ozawa \cite{ozawa} showed that $S_A$ is an essential surface in $S^3 \setminus K$. A different proof is given in \cite{fkp:gutsjp}.

\begin{define}
Let $M_A := S^3 \setminus N(S_A)$ denote the complement of an open regular 
neighborhood of $S_A$. When convenient, we will also use the shorthand notation $S^3 \cut S_A$  instead of $S^3 \setminus N(S_A)$.  The boundary of $M_A$ decomposes into the \emph{parabolic locus} (the remnants in $M_A$ of the boundary tori of $S^3 \setminus K$), and a surface $\widetilde{S_A}$ that can be identified as the frontier of $N(S_A)$ in $S^3 \setminus K$. Note that $\widetilde{S_A}$ is a double cover of $S_A$, connected if and only if $S_A$ is non-orientable.
\end{define}

The main technical tool of \cite{fkp:gutsjp} that we use is a
decomposition of $M_A$ into ideal polyhedra.  The faces of these
polyhedra are checkerboard colored, white and shaded.  The white faces
of each polyhedron are glued to another polyhedron, while the shaded
faces lie on $\widetilde{S_A}$.  There is exactly one \emph{upper
  polyhedron}, which occupies the $3$--ball above the projection
plane.  There are multiple \emph{lower polyhedra}, each of which is
glued along its white faces to the upper polyhedron only.  All the
polyhedra are \emph{prime}, in the sense that a pair of faces share at
most one edge.

The precise combinatorics of the ideal polyhedra can be read off from
the diagram $D(K)$ and the graph $H_A$. (See \cite[Chapter 2]{fkp:gutsjp} 
and \cite[Section 5]{fkp:survey} for 
details on how to do this.)  Here, we describe the features that
are salient for this paper.  

The first feature we will need is information about the combinatorics of the lower polyhedra.  This information comes from subgraphs of $H_A$, or slight modifications of subgraphs, which we call \emph{polyhedral regions}. Their precise definition is as follows.

\begin{define}\label{def:polyhedral-region}
Suppose $\alpha$ is an arc in the complement of $H_A$ with both endpoints on a state circle $C$.  Consider the subgraph of $H_A$ consisting of $C$ and all state circles and segments which lie on the same side of $C$ as $\alpha$.  Note that $\alpha$ cuts the subgraph into two components, one on either side of $\alpha$. If both components contain segments, then we say $\alpha$ is a
\emph{non-prime arc}.  A collection of non-prime arcs is \emph{maximal} if, once we cut along all such arcs and all state circles, the graph decomposes into subgraphs that each contain a segment, and no larger collection of non-prime arcs has this property.  Figure~\ref{fig:polyregex}, left, shows an example of a graph $H_A$ with a maximal collection of non-prime arcs. 

Let $\{\alpha_1, \dots, \alpha_n\}$ denote a maximal collection of non-prime arcs.  A  \emph{polyhedral region} is a nontrivial region of the complement of the state circles and the $\alpha_i$, where  by \emph{nontrivial} we mean the region contains segments. Each lower polyhedron corresponds to precisely one of these polyhedral regions. Note that if $H_A$ admits no non-prime arcs, then a polyhedral region is just a region of the complement of the state circles which contains segments.  Figure~\ref{fig:polyregex} shows an example.
\end{define}

\begin{figure}
  \includegraphics[scale=.9]{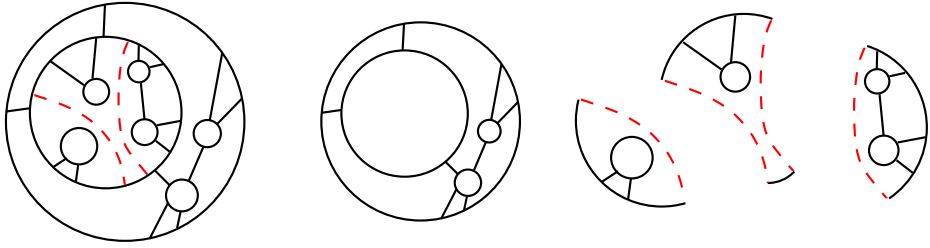}
  \caption{Left: A graph $H_A$ with a maximal collection of non-prime arcs (shown in dashed red).  Right: This example breaks into four distinct polyhedral regions, as shown.}
  \label{fig:polyregex}
\end{figure}

Each white face of the polyhedra corresponds to a
nontrivial (i.e.\ non-innermost disk) complementary region of $H_A
\cup (\cup_{i=1}^n \alpha_i)$.  The white faces that belong to a
lower polyhedron are glued to corresponding white faces in the unique upper
polyhedron.

Associated to each polyhedral region $R$, and a corresponding lower
polyhedron $P$, we have a \emph{clockwise map}. Loosely speaking, the
clockwise map $\phi$ gives us a way to associate the lower polyhedron
$P$ with a section of the upper polyhedron.

\begin{define}\label{def:clockwise}
The \emph{clockwise map} $\clock$ is a homeomorphism from the white faces of
the upper polyhedron belonging to the polyhedral region $R$ to the white faces of the corresponding lower polyhedron $P$.  On each white face, $\clock$ is defined by composing the gluing map to a white face of $P$ with a single clockwise rotation in that face.

Figure~\ref{fig:clockwise} shows an example of how the clockwise map compares to the gluing map.  
\end{define}

\begin{figure}
  \input{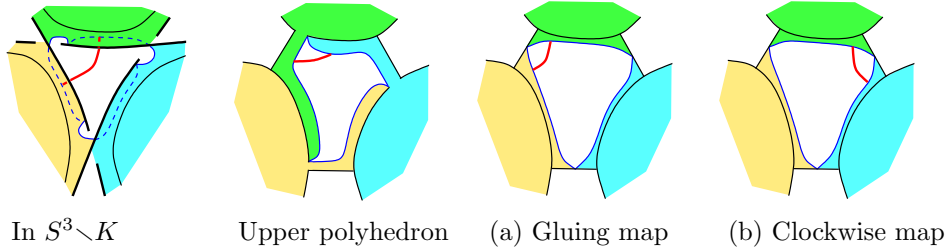}
  \caption{An edge (in red) in the link complement is shown, along with its position in the upper polyhedron, and images under the gluing and clockwise maps. 
  }
  \label{fig:clockwise}
\end{figure}

There is a way to extend the domain of definition of the clockwise map to \emph{normal squares}, that is, normal disks with $4$ sides, and we will use this multiple times in this paper. The following is a restatement of \cite[Lemma 4.8]{fkp:gutsjp}.

\begin{lemma}\label{lemma:clockwise}
Let $S$ be a normal square in the upper polyhedron, with arcs
$\beta_V$ and $\beta_W$ in white faces $V,W$ that belong to a
polyhedral region $R$. Let $P$ be the lower polyhedron corresponding
to $R$, with clockwise map $\phi$.  Then there is a normal square
$\phi(S) \subset P$, unique up to normal isotopy, which contains white
sides $\phi(\beta_V)$ and $\phi(\beta_W)$.

Furthermore, if $S$ is glued along $V$ to a square $T$ in the lower polyhedron, then $\phi(S) \cap V$ will differ from $T \cap V$ by a single clockwise rotation of $V$.
\qed
\end{lemma}

%%%%%%%%%%%%%%%%%%%%%%%%%%%%%%%%%%%%%%%%%%%%%%%%%%%%%%%%%%%%%%%%%
%%%%%%%%%%%%%%%%%%%%%%%%%%%%%%%%%%%%%%%%%%%%%%%%%%%%%%%%%%%%%%%%%
\section{Tori, annuli, and squares}\label{sec:annuli}
In this section, we will consider essential tori embedded in the
complement of a semi-adequate knot.  We will see that the state
surface $S_A$ cuts these into annuli, and we will consider properties
of these annuli.  In particular, we will show that the annuli
decompose into squares, all of which are either diagrammatically
compressible or diagrammatically incompressible.  This sets the stage
for the next two sections, in which these two cases are analyzed
separately.

\begin{lemma}\label{lemma:tori-cutto-annuli}
Let $D(K)$ be an $A$--adequate diagram, with all--$A$ state
surface $S_A$.  Let $T$ be an essential torus.  Then we may
isotope $T$ such that $T\setminus \widetilde{S_A}$ is an even number of
essential annuli, half embedded in $M_A$ and half embedded in the
$I$--bundle $N(S_A)$.

Furthermore, those essential annuli in $M_A$ are cut into normal
squares by the white faces of the polyhedral decomposition of $M_A$,
where each square has two opposite sides on shaded faces and two
opposite sides on white faces.
\end{lemma}

\begin{proof}
Isotope $T$ to be transverse to $S_A$ and minimize the number of
curves of intersection with $S_A$.  Because $S_A$ and $T$ are both
essential, this ensures that all intersections $T \cap S_A$ are
nontrivial simple closed curves on $T$.  Hence $T \cap N(S_A)$ and $T
\setminus N(S_A)$ consist of annuli, which are essential because $T$
is essential.

Because the annuli in the closure of $T\setminus N(S_A)$ in $M_A$ are
essential, we can put them into normal form with respect to the
polyhedral decomposition of $M_A$.  This may involve isotopy of the
boundary components of the annuli.  We may isotope the adjacent annuli in
$N(S_A)$ to ensure that when we isotope annuli
into normal form, we actually isotope the entire torus.

Let $E \subset T \cap M_A$ be an essential annulus in normal
form. Since $E$ cannot be contained in a single polyhedron (because it
is essential), it must intersect the white faces. We claim that no arc
of intersection between $E$ and a white face can be parallel to the
boundary of $E$. This is because an outermost such arc would cut off a
normal bigon in an ideal polyhedron, and our \emph{prime} ideal
polyhedra do not contain normal bigons \cite[Proposition
  3.18]{fkp:gutsjp}. Therefore, every arc of intersection between $E$
and a white face runs across $E$, from one boundary circle to the
other. These arcs cut $E$ into normal squares, finishing the lemma.
\end{proof}

We will investigate the annuli and squares of Lemma
\ref{lemma:tori-cutto-annuli}.  The study of these surfaces will
naturally break into two cases: whether the squares cut off a single
ideal vertex in a white face, or whether each edge in a white face
cuts off multiple ideal vertices on both sides.  This is encoded in
the following definition.

\begin{define}\label{def:diag-comp}
Let $S \subset M_A$ be a surface in normal form. We say that $S$ is  \emph{diagrammatically compressible} if, in some white face $W$ of the  polyhedral decomposition, an arc of $S \cap W$ runs between two  adjacent edges of $W$. In other words, $S$ is diagrammatically  compressible in $W$ if $S \cap W$ cuts off an ideal vertex of $W$.   See Figure~\ref{fig:diagramincompr}.   Otherwise, if no such white face exists, we call $S$  \emph{diagrammatically incompressible}.
\end{define}

\begin{figure}
  \includegraphics{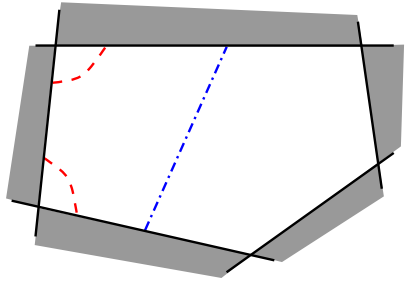}
  \caption{Shown is a white face $W$ of the polyhedral decomposition.  Red arcs (dashed) cut off a single ideal vertex of $W$.  A surface meeting $W$ in such an arc is diagrammatically compressible.  A surface meeting $W$ in the blue arc (dot-dashed) is diagrammatically incompressible, provided its intersections with other white faces also do not cut off a single ideal vertex.}
  \label{fig:diagramincompr}
\end{figure}

If $S$ is diagrammatically compressible in some white face $W$, it
cuts off a disk $U \subset W$ with two sides on shaded faces, one side
on $S$, and one side on an ideal vertex.  Such a disk $U$ is an
example of what Lackenby calls a \emph{parabolic compression disk}.
In other words, diagrammatically compressible surfaces are also
\emph{parabolically compressible} (see \cite[Page
  209]{lackenby:volume-alt} for a definition).  For annuli, the
converse also holds: by \cite[Proposition 4.21]{fkp:gutsjp}, a
parabolically compressible annulus $A \subset M_A$ must be
diagrammatically compressible.  Because we will be working with annuli
below, we will not need the notion of parabolic compressibility in
this paper.

\begin{lemma}\label{lemma:different-regions}
Let $Q$ be a square in the upper polyhedron, glued to normal squares
in lower polyhedra at each of its white sides. If the white sides of
$Q$ come from different polyhedral regions, then they each cut off a
single ideal vertex in that white face.  Hence $Q$ diagrammatically
compresses in both of its white faces in the upper polyhedron.
\end{lemma}

\begin{proof}
 This is a restatement of \cite[Proposition 4.13]{fkp:gutsjp}.
\end{proof}

In this paper, normal squares appear in decompositions of essential
annuli.  A square $Q$ in the upper polyhedron will be glued to normal
squares in the lower polyhedra at each of its white sides, and so
Lemma \ref{lemma:different-regions} will be useful.
%whenever $Q$ is a normal square in the decomposition of an essential
%annulus.

The next two sections include two cases.  First, that the squares
making up the annulus are diagrammatically incompressible, and second,
that they are diagrammatically compressible.  Lemma
\ref{lemma:different-regions} puts restrictions on the
diagrammatically incompressible case, and so we investigate such
annuli first.

%%%%%%%%%%%%%%%%%%%%%%%%%%%%%%%%%%%%%%%%%%%%%%%%%%%%%%%%%%%%%%%%%
%%%%%%%%%%%%%%%%%%%%%%%%%%%%%%%%%%%%%%%%%%%%%%%%%%%%%%%%%%%%%%%%%
\section{Diagrammatically incompressible annuli}\label{sec:incompr}
In this section, we determine the form of any essential annulus that
intersects a polyhedral region in a diagrammatically incompressible
way. Then, in Lemma \ref{lemma:incomprannuli}, we determine how diagrammatically incompressible annuli 
can fit into essential tori in the link complement.
 We emphasize  that all of the results of this section work for
general $A$--adequate diagrams, without any extra hypotheses.

Following Lackenby \cite{lackenby:volume-alt}, we define a \emph{fused
  unit} to be a portion of a checkerboard colored graph with the
following property.  Its boundary is an essential square, with two
opposite sides in black regions and the other two sides each
intersecting a white region adjacent to a crossing, as on the left of
Figure \ref{fig:fused-unit}.  In the figure, the question marks can
represent any checkerboard graph corresponding to an alternating
tangle.  

\begin{figure}
  \includegraphics{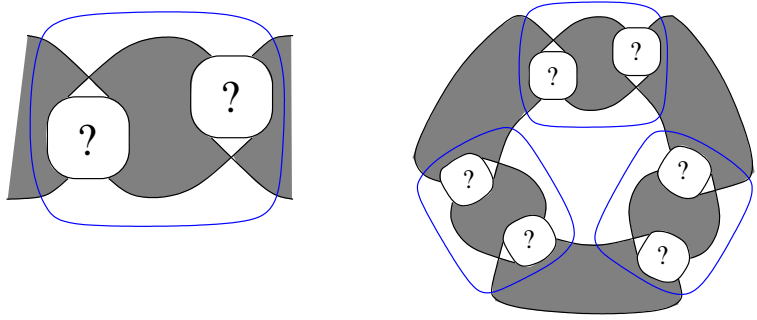}
  \caption{ A fused unit and an example of a cycle of three fused units.}
  \label{fig:fused-unit}
\end{figure}

\begin{lemma}\label{lemma:para-incompr-lower}
Let $D(K)$ be an $A$--adequate diagram, and let
$(E, \bdy E)\subset (M_A, \widetilde{S_A})$
be a diagrammatically incompressible annulus.
Then $E$ lies in a single polyhedral region, and the lower polyhedron
in that region is a cycle of $n \geq 2$ fused units.  Moreover,
portions of $E$ that lie in the lower polyhedron are squares
encircling a fused unit, as the blue curves in Figure
\ref{fig:fused-unit}.
\end{lemma}

\begin{proof}
Put $E$ into normal form with respect to the polyhedral
decomposition of $M_A$.  This will cut $E$ into a sequence of
squares, alternating in upper and lower polyhedra.  

Consider a square $S \subset E$ in the upper polyhedron. If the two
white sides of $S$ lie in different polyhedral regions, Lemma
\ref{lemma:different-regions} implies that it cuts off an ideal vertex
in some white face, contradicting the hypothesis of diagrammatic
incompressibility.  Hence all squares of $E$ have their white sides in
a single polyhedral region.

For a square $S \subset E$ in the upper polyhedron, Lemma
\ref{lemma:clockwise} lets us apply the clockwise map and obtain a
square $\phi(S)$ in the lower polyhedron.  Consider these squares, as
well as the squares of $E$ originally in the lower polyhedron.  Label
the squares in the lower polyhedron $S_1, S_2, \dots, S_{2n}$, where
$S_i$ with even $i$ are the clockwise images of squares from the upper
polyhedron.

By \cite[Lemma 4.10]{fkp:gutsjp}, if any pair of adjacent squares
$S_k$ and $S_{k+1}$ does not have intersecting white sides, then those
squares must cut off single vertices in each of their white sides,
implying both are diagrammatically compressible to essential product
disks, which again contradicts the hypotheses.

So suppose that each $S_k$ and $S_{k+1}$ intersect in one, hence by
\cite[Lemma 4.10]{fkp:gutsjp}, both white sides.  
By Lemma \ref{lemma:clockwise}, the white
sides of odd-numbered squares differ from those of even-numbered squares by a clockwise
rotation. 
This allows us to sketch the form of the diagram, essentially following
Lackenby's proof of \cite[Theorem 14]{lackenby:volume-alt}.

First, suppose there are just two squares $S_1$ and $S_2$.  Then they
intersect in exactly two white faces of the same lower
polyhedron.  In each white face, $S_2$ differs from $S_1$ by a single
clockwise rotation, as shown in Figure \ref{fig:2squares}, left.  Moreover,
the white sides of the squares glue up with orientations shown in that
figure.  But now, note that these two squares glue to form a
M\"obius band, not an annulus, which is a contradiction.

\begin{figure}
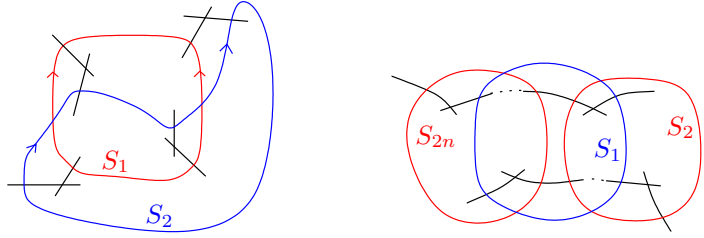

 \input{figures/2squares.pstex_t}
 \hspace{.5in}
 \input{figures/para-incompress.pstex_t}
 \caption{Left: If annulus is formed of only two squares. Right:
   Squares $S_1$, $S_2$, and $S_{2n}$ must be as shown.}
 \label{fig:2squares}
\end{figure}

So there must be at least four squares.  Then $S_1$ differs from $S_2$
in one of its white faces by a single clockwise rotation, and $S_1$
differs from $S_{2n}$ in the other white face of $S_1$ by a single
clockwise rotation.  The fact that $S_2$ and $S_{2n}$ are disjoint
\cite[Lemma 4.8 (3)]{fkp:gutsjp} implies that they must lie in the
lower polyhedron as shown in Figure \ref{fig:2squares}, right.  Note
this implies that $S_1$ bounds a fused unit.

Applying the same argument to $S_3$, $S_4$, and $S_2$, and continuing
through all squares with odd indices, we find each odd square bounds a
fused unit, and these are arranged in a cycle as claimed in the
lemma.
\end{proof}

\begin{lemma}\label{lemma:para-incompr-upper}
Let $D(K)$ be an $A$--adequate diagram.  Suppose $(E, \bdy E) \subset
(M_A, \widetilde{S_A})$ is an embedded essential annulus such that $E$
is diagrammatically incompressible. Then there is a solid torus
$V\subset M_A$ whose boundary consists of $E$ and an annulus $F
\subset \widetilde{S_A}$.

Furthermore, each of the annuli $E,F \subset \bdy V$ winds once around
the meridian of $V$ and $n$ times around the longitude of $V$, for the
same integer $n \geq 2$ as in Lemma \ref{lemma:para-incompr-lower}. In
other words, the two curves of $E \cap F$ have slope $1/n$ on the
boundary of $V$.
\end{lemma}

\begin{proof}
By Lemma \ref{lemma:tori-cutto-annuli}, white faces of the polyhedral
decomposition of $M_A$ cut $E$ into squares, alternating between lying
in the upper and lower polyhedra.  Denote the squares by $S_1, S_2,
\dots, S_{2n}$, where $S_i$ for even $i$ lies in the upper polyhedron.
(Note in the previous proof, $S_i$ for even $i$ indicated the
\emph{images} of these squares in the lower polyhedron.)

By Lemma \ref{lemma:para-incompr-lower}, the lower polyhedron $P$ associated to $E$ 
 is a cycle of at
least two fused units, as in Figure \ref{fig:fused-unit}.  Moreover,
the squares $S_i$ for $i$ odd encircle a fused unit.

Now recall from Section \ref{subsec:polyhedra} that a lower polyhedron
corresponds to a polyhedral region, i.e.\ a nontrivial region of the
complement of the state circles and a maximal collection of non-prime
arcs.  A white face in a lower polyhedron $P$ is glued via homeomorphism
to exactly one white face in the upper polyhedron.  However, shaded
faces are not glued.  In a region of the upper polyhedron corresponding to a
shaded face of $P$, there may be additional segments and state circles in the
graph $H_A$.  However, we may use information on white faces, as well
as positions of state circles appearing in the given lower polyhedron,
to sketch portions of the graph $H_A$.   Figure \ref{fig:upper-paraincomp}
shows the most general possible
graph.

\begin{figure}[h]
  \includegraphics{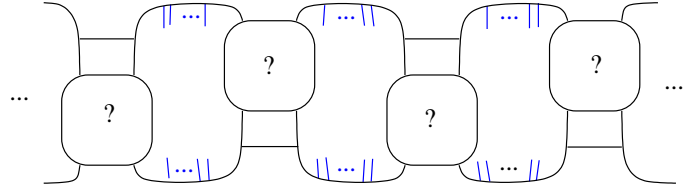}
  \caption{The graph $H_A$ must have the form shown when one of
    the lower polyhedra is a cycle of fused units.}
  \label{fig:upper-paraincomp}
\end{figure}

Note in particular that the white faces above and below the diagram in
Figure \ref{fig:upper-paraincomp} are mapped to the white faces inside
and outside the cycle of fused units.  Hence they only meet segments
of $H_A$ that correspond to ideal vertices of the lower polyhedron.  Thus
the only possible segments in these white faces are either between
state circles as shown, or inside the blocks labeled with question
marks.  However, \emph{a priori}, there may be segments on the other
sides of the state circles meeting these two white faces.  Such
segments are illustrated in blue in Figure \ref{fig:upper-paraincomp}.

The combinatorics of the upper polyhedron can be read off of the graph $H_A$, as described in \cite[Chapter 2]{fkp:gutsjp} or \cite{fkp:survey}.  What is relevant to this discussion is that shaded faces run along segments in so--called \emph{tentacles}.  A tentacle is a portion of shaded face that begins on one side of a state circle (the ``head''), runs along the right side of a segment when the head is oriented to be up, and then runs to the right along the adjacent state circle until it terminates at a segment.  Examples of tentacles in different colors are shown in Figure~\ref{fig:upper-pinc-wSi}.

The boundaries of these tentacles make up the edges of the white faces in our fused units.  Hence, we use what we know of the positions of white sides of $S_i$ for $i$ odd to sketch white sides of $S_j$ for $j$ even into the upper polyhedron.  In particular, white sides in the lower polyhedron are glued to those in the upper, with $S_2$ glued to $S_1$ on one side, and $S_3$ on the other side.  Therefore, these edges are as shown in Figure~\ref{fig:upper-pinc-wSi}.  A pair of endpoints of white edges connect to an edge in a shaded face.  Thus we may color the shaded faces at the ends of white edges in the same color.  This is also shown in Figure~\ref{fig:upper-pinc-wSi}.

\begin{figure}
  \input{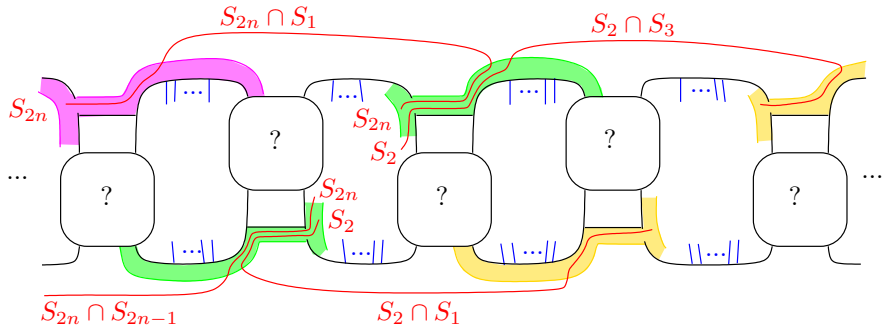}
  \caption{Upper polyhedron with squares $S_{2k}$ is as above.}
  \label{fig:upper-pinc-wSi}
\end{figure}

Consider the shaded face shown in the center of Figure~\ref{fig:upper-pinc-wSi}, which is shaded green.  Note that $S_{2n}$ and $S_2$ must both run through this green face, and both meet the same pair of tentacles on either side.  Because shaded faces are simply connected \cite[Theorem~3.12]{fkp:gutsjp}, these arcs on $S_2$ and $S_{2n}$ must run parallel to each other, bounding a strip of green shaded face between them.  The same argument implies $S_{2k}$ and $S_{2(k+1)}$ bound a strip of shaded face between them, for all $k$.

Now, notice that squares $S_i$ with even index cut out a prism from the upper polyhedron. This prism has $n$ sides coming from the $S_{2k}$, $n$ sides coming from these strips of shaded face, and a top and bottom white face, each a polygon with $2n$ sides.

From Lemma~\ref{lemma:para-incompr-lower} and Figure~\ref{fig:fused-unit}, 
we see that the $S_i$ with odd index also cut out a prism from the lower polyhedron, with $n$ sides coming from squares $S_{2i+1}$, and $n$ sides coming from strips of shaded face, and a top and bottom white face, each a polygon with $2n$ sides.

These two prisms glue along the white faces to give a solid torus $V$, with boundary consisting of the annulus $E$, as well as an annulus $F \subset \widetilde{S_A}$ obtained by gluing the strips coming from  shaded faces end to end. The white faces of the prisms form meridian  disks of the solid torus $V$.

Observe from Figure~\ref{fig:upper-pinc-wSi} that the annulus $E$ intersects each meridian disk $n$ times. Observe as well that as we travel around $E$, the square $S_{2n}$ in the upper polyhedron is glued to $S_1$ in the lower, then $S_2$ in the upper polyhedron, and so on --- with $S_{2n}$ differing from $S_2$ by a single $1/n$ clockwise twist of the white face. Therefore, the annulus $E$ composed of these squares will intersect the longitude of $V$ once. Since the core  curve of $E$ intersects the meridian $n$ times and the longitude once, with clockwise twisting, its slope on $\bdy V$ is $1/ n$. The core curve of $F$ is parallel to that of $E$, and has the same slope.
\end{proof}

Now recall that we are interested in an essential torus $T$, which is
cut into annuli in $M_A$ by Lemma \ref{lemma:tori-cutto-annuli}.

\begin{define}\label{def:annulus-adjacent}
Let $T \subset S^3 \setminus K$ be an essential torus, decomposed into
annuli as in Lemma \ref{lemma:tori-cutto-annuli}.  We say that an
annulus $E \subset M_A$ is \emph{adjacent} to $E' \subset M_A$ if
there is a single annulus in $N(S_A)$ between them.
\end{define}

\begin{lemma}\label{lemma:incomprannuli}
Let $D(K)$ be an $A$--adequate diagram.  Let $T \subset S^3 \setminus K$ be
an essential torus in $S^3\setminus K$, and let $E, E' \subset T
\cap M_A$ be adjacent annuli.  Then at least one of $E, E'$ must be
diagrammatically compressible.  In particular, if $E$ is adjacent to
itself, then it must be diagrammatically compressible.
\end{lemma}

\begin{proof}
Suppose $E$ is diagrammatically incompressible.  We will examine 
more closely the solid torus $V$ and the annulus $F$ of Lemma
\ref{lemma:para-incompr-upper}.  Since $T$ is a torus in $S^3$, it is
separating.  Let $X$ be the component of $(S^3\setminus K)\cut T$ that
contains $V$ and $X\cut S_A$ denote the remnants of $X$ in $S^3\cut S_A$.
The manifold $X\cut S_A$ contains $V$ as one of its
components.  The annulus $F \subset \widetilde{S_A} \cut T$
projects to some component $R$ of $S_A\cut T$.
\smallskip

\underline{Case 1.}  Suppose $R$ is orientable.  Then it is an
annulus.  When we glue $\widetilde{S_A}$ to itself, to undo the
cutting along $S_A$ and recover $S^3\setminus K$, the annulus $F
\subset \widetilde{S_A} \cap \bdy V$ must be glued to another annulus
$F'$, which is the boundary of some other component of $X \cut S_A$.
Assuming that the adjacent annulus $E' \subset T \cap M_A $ is also
diagrammatically incompressible, Lemma \ref{lemma:para-incompr-upper}
implies there must be a solid torus $V' \subset X \cut S_A$, with
boundary $\bdy V' = E' \cup F'$, such that the core curve of $E'$ has
slope $1/n$.  By Lemma \ref{lemma:para-incompr-lower}, $n$ is at least $2$.
But then $\bdy E$ and $\bdy E'$ are both glued onto
$\bdy R$, hence $E$ and $E'$ form all of $T \cap M_A$.  Hence $X$
consists of exactly two solid tori, glued along annuli of slopes
$1/n_i$ on their respective boundaries.

\smallskip

\underline{Case 2.} Suppose $R$ is non-orientable.  Then, since $F
\subset \bdy N(R) \subset \widetilde{S_A} $ is a double cover of $R$,
it follows that $R$ must be a M\"obius band.  To form $S^3 \setminus
K$ from $S^3 \cut S_A$, we glue appropriate boundary components of
$N(S_A)$.  The annulus $F$, as the boundary of a regular neighborhood
of a M\"obius band, must be glued to itself, with the regular
neighborhood $N(R)$ collapsing onto $R$.  But then under this gluing,
the boundary components of the annulus $E$ are glued only to boundary
components of the annulus $E$.  This means $S_A$ cuts $T$ into only
the annulus $E$, and $V$ is the only component of $X \cut S_A$.
Therefore, $X$ must be the result of gluing $V$ to itself along an
annulus with slope $1/n$ on its boundary.

Here is another way to think of this gluing. The regular neighborhood
$N(R) \subset N(S_A)$ is a solid torus.  Since $F$ double-covers the
M\"obius band $R$ and runs around this solid torus twice, the slope of
$F$ on $\bdy N(R)$ is $1/2$.  We conclude that $X$ is obtained by
gluing together two solid tori along annuli whose slopes are $1/n$ and
$1/2$.

\smallskip

In both cases, $X$ consists of two solid tori $V_1$ and $V_2$, glued
along annuli of slopes $1/n_i$, where each $n_i$ is at least $2$.
This means that $X$ is a Seifert fibered space with base space a disk
and two singular fibers. (See Hatcher \cite{hatcher:3manifolds} for
background.)

Next, we will calculate the Seifert invariants of $X$.  For simple,
oriented curves $x,y$ on $T_i = \bdy V_i$, let $\langle x, y \rangle$
denote their algebraic intersection number. Also abusing notation,
we will use the same symbol to denote a curve on $T_i$ and its
homology class in $H_1(T_i)$.  Let $Q_i$, $H_i$ denote a pair of a
cross-section curve and fiber of the fibration on $T_i$, oriented so
that $\langle Q_i , H_i \rangle=1$.  Let $\mu_i, \lambda_i$ be a
meridian and longitude of $T_i$, chosen so that $\mu_i$ bounds a disk
in $V_i$ and $\langle \mu_i , \lambda_i \rangle= 1$.  Furthermore, we
can normalize $\lambda_i$ so that the gluing annulus, which is
foliated by fibers parallel to $H_i$, satisfies $H_i= \mu_i + n_i
\lambda_i$.  Since $Q_i=a_i \mu_i + b_i \lambda_i$ for some
coefficients $a_i, b_i$, it follows that
 \begin{equation}\label{eqn:indiv-slope}
\langle Q_i , H_i \rangle =   n_i a_i-  b_i = 1.
 \end{equation} 
Solving for $\mu_i$, we have $\mu_i= n_i Q_i -b_i H_i$. Thus the slope of each singular fiber of $X$  
is $-b_i/n_i$, making the Euler number of  $X$ equal to
\begin{equation}\label{eqn:euler-num}
e = -\frac{b_1}{n_1} - \frac{b_2}{n_2} = \frac{-n_2 b_1 - b_2 n_1}{n_1 n_2}.
\end{equation}
 
Since $T = \bdy X$ is incompressible, $X$ cannot be a solid torus.
Thus $X$ admits a unique Seifert fibration \cite[Theorem
  2.3]{hatcher:3manifolds},
which means that its Seifert invariants are uniquely
determined modulo $1$ \cite[Proposition 2.1]{hatcher:3manifolds}.
Since $X$ embeds in $S^3$ and $\bdy X$ is a single
torus, $X$ must be the complement of a torus knot, say $H$. On $\bdy X$, $H$ may be identified with a regular fiber of the fibration of $X$.
Let $\mu, \lambda$ denote the meridian and canonical longitude of $\bdy X$, again with the convention that $\langle \mu , \lambda \rangle=1$.
Since the geometric intersection of the meridian and $H$ is 1, if $Q$ is a cross section curve on $\bdy X$, then
\begin{equation}\label{eqn:mu-slope}
\mu=Q + xH, 
\end{equation}
so that $\langle \mu,  H \rangle = \langle Q , H \rangle=1$.
A minimum  genus Seifert surface for $\lambda$  will be horizontal with
respect to the Seifert fibration of $X$. By the Claim in the proof of
\cite[Proposition 2.2]{hatcher:3manifolds},  the boundary slope of this surface with respect to the $(Q,H)$ framing is equal to the Euler number $e$. Therefore, by equation \eqref{eqn:euler-num},
\begin{equation}\label{eqn:lambda-slope}
\lambda=  n_1n_2 Q + (-n_2b_1-b_2n_1)  H.
\end{equation}

Since $\langle \mu , \lambda \rangle=1$,  equations  \eqref{eqn:mu-slope}  and \eqref{eqn:lambda-slope} imply
\begin{equation}\label{eqn:xnb}
-x n_1  n_2-n_2b_1-b_2n_1=1.
\end{equation}
 By \eqref{eqn:indiv-slope}, we obtain
  $n_1n_2a_1+n_2n_1a_2- (n_1b_2+b_1n_2) = (n_1+n_2) $,
  which combined with \eqref{eqn:xnb} gives
    $$n_1n_2(a_1+a_2+x)+ 1=(n_1+ n_2)  
\quad \Rightarrow \quad
a_1+a_2+x= 1/n_1+1/n_2-1/{n_1n_2}\, ,$$
which is impossible since  $a_i, x$ are integers and $n_i \geq 2$.
 \end{proof}

\begin{remark}
Lemma \ref{lemma:incomprannuli} can also be proved using diagrammatic
techniques.  The idea is as follows.  From Figures
\ref{fig:fused-unit} and \ref{fig:upper-pinc-wSi}, one can determine
how the annulus $F\subset \widetilde{S_A}$ lies in the diagram.  If
two annuli $F$ and $F'$ as above are mapped to a single annulus in
$S_A$ (Case 1 of the above proof), then their cores must map to the
same curve on $S_A$.  A careful analysis of annuli $F$ and $F'$, each
lying in the diagram as specified by Figures \ref{fig:fused-unit} and
\ref{fig:upper-pinc-wSi}, yields a contradiction.  Similarly, if $F$
is glued to itself when $\widetilde{S_A}$ is mapped to $S_A$ (Case 2
of the above proof), then the core of $F$ must wrap around the same
curve on $S_A$ twice.  Again a diagrammatic analysis will reveal that
this is impossible.
 \end{remark}

%%%%%%%%%%%%%%%%%%%%%%%%%%%%%%%%%%%%%%%%%%%%%%%%%%%%%%%%%%%%%%%%%
%%%%%%%%%%%%%%%%%%%%%%%%%%%%%%%%%%%%%%%%%%%%%%%%%%%%%%%%%%%%%%%%%
\section{Diagrammatically compressible annuli}\label{sec:compr}
Lemma \ref{lemma:incomprannuli} implies that  the intersection of an essential torus 
with 
the corresponding polyhedral decomposition must contain components that intersect the decomposition in a diagrammatically compressible way.
In this section, we determine information on  such annuli under the hypothesis of Theorem \ref{thm:main}.

Suppose an annulus decomposes into a square that is diagrammatically
compressible in one white face, i.e.\ it cuts off a single ideal
vertex there.  Then it must be diagrammatically compressible in every
white face, either by Lemma \ref{lemma:different-regions}, or by an
application of the clockwise map, Lemma \ref{lemma:clockwise}, and
\cite[Lemma 4.10]{fkp:gutsjp}.  Diagrammatically compressible squares
are closely related to essential product disks.

\begin{define}\label{def:EPD}
An \emph{essential product disk}, or \emph{EPD}, is a properly
embedded essential disk in $M_A$ whose boundary meets the parabolic
locus of $M_A$ twice.  
\end{define}

When an essential product disk $R$ lies in a single polyhedron in the
polyhedral decomposition of $M_A$, we may think of it as a
quadrilateral with two sides on shaded faces, coming from
$\widetilde{S_A}$, and two sides running over ideal vertices of the
polyhedron, which correspond to the parabolic locus.  We may pull $R$
off the ideal vertices into adjacent white faces, obtaining a normal
square $Q$, with two sides on shaded faces and two sides on white
faces.  The two sides of $Q$ on white faces each cut off a single
ideal vertex; that is  the square is diagrammatically compressible in both
of its white faces.  Conversely, if $Q$ is a normal square, each of
whose white sides cuts off a single ideal vertex of the ambient white
face, then pulling $Q$ onto those ideal vertices (that is, performing a
\emph{parabolic compression})  produces either an EPD or a
square encircling a single ideal vertex.

\begin{remark}\label{remark:simple}
Throughout this section, we will assume that the sides of normal squares in shaded faces run
monotonically through tentacles, without unnecessary
backtracking. This assumption can be easily satisfied by normal isotopy. The precise terminology from \cite{fkp:gutsjp} is that the sides of squares are \emph{simple with respect to the shaded faces}, as in
\cite[Definition 3.2]{fkp:gutsjp}. 
\end{remark}

\begin{convention}\label{conv:coloring}
Let $Q$ be a normal square, and $W$ a white face in which $Q$ cuts off a single ideal vertex. We may color the shaded faces met by $Q$ \emph{orange} and \emph{green}, so that the single vertex of $W$ cut off by $Q$ is a triangle whose three edges, in counter--clockwise order, are orange--green--white. For instance, the right--most white face in Figure \ref{fig:EPD} satisfies this convention.
\end{convention}

The next lemma places strict restrictions on the form of a square coming from
an EPD in the upper polyhedron.  The proof relies on the hypothesis that 2--edge loops in
$\GA$ belong to twist regions (Definition \ref{def:2edges-belong}), as well as the classification of EPDs into seven combinatorial types  \cite[Theorem~6.4]{fkp:gutsjp}. This proof is likely the most technical argument of this paper, and can be omitted without missing the thread of the argument.

%%%%%%%%%%%%%%%%%%%%%%%%%%%%%%%%%%%%%%%%%%%%%%%%%%%%%%%%%%%%%%%%%

\begin{figure}
  \input{figures/squareA.pstex_t}
 \hspace{.3in}
 \includegraphics{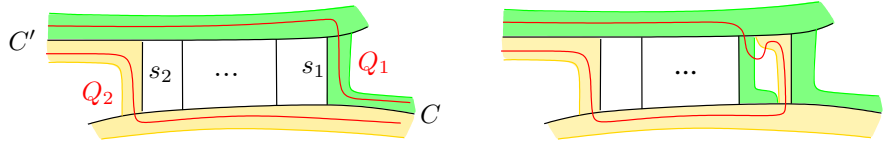}
\caption{Possible forms of squares parabolically compressing to an
  EPD.  The square on the left is called a square in a twist region
  (Definition \ref{def:square-in-twist}). The right--most white face satisfies Convention~\ref{conv:coloring}.}
  \label{fig:EPD}
\end{figure}

\begin{lemma}\label{lemma:EPD-typeB}
Suppose a prime, $A$--adequate diagram is such that $\GA$ satisfies
the 2--edge loop condition.  Suppose that a normal square $Q$
parabolically compresses to an EPD in the upper polyhedron.  Then,
possibly after sliding one white side of $Q$ past an ideal vertex, we
obtain a normal square $P$ with one of the two forms shown in Figure
\ref{fig:EPD}.  In particular, the boundary of $P$ runs over distinct
segments of the same twist region, with two sides on shaded faces
adjacent to the same state circle on a side of that twist region.
\end{lemma}

\begin{proof}
Given a diagrammatically compressible square $Q$, select one of the
sides that cuts off a single ideal vertex in a white face $W$.  Color the
two adjacent shaded faces according to Convention~\ref{conv:coloring}.

With this labeling of shaded faces, consider the side of $Q$ in its other white
face, $W'$.  This side also cuts off a single ideal vertex, but that ideal vertex may be a triangle with opposite orientation compared to that of Convention~\ref{conv:coloring}.  If this occurs, replace $Q$ with a new square $P$ by sliding the white side $Q \cap W'$ across the adjacent ideal vertex, so that it lies in a new white face $W''$ and cuts off a triangle oriented in the opposite direction.  This new square $P$ must satisfy Convention~\ref{conv:coloring} in both of its white faces.

We now apply \cite[Theorem 6.4]{fkp:gutsjp}, which applies to normal squares such as $P$ that satisfy Convention~\ref{conv:coloring} in both white faces (equivalently, satisfy~\cite[Lemma~6.1]{fkp:gutsjp}). This theorem implies that the normal square $P$ is of one of seven types, shown in~\cite[Figure~6.1]{fkp:gutsjp}.  Assuming that 2--edge loops of $\GA$ belong to twist regions, we may rule out all these types except the first two, labeled $\mathcal{A}$ and $\mathcal{B}$ in~\cite[Figure~6.1]{fkp:gutsjp}.  These are exactly the two types shown in Figure~\ref{fig:EPD}.  We now describe how to rule out the remaining types.

\smallskip
{\underline{Type $\mathcal{C}$.} }  Suppose the normal square is of type $\mathcal{C}$, reproduced in Figure~\ref{fig:typeC}, left.  The 2--edge loop shown belongs to a twist region by hypothesis, and so a twist region must lie either on the inside of the segments shown or the outside.  Note there is a segment meeting the lower state circle on the opposite side of the 2--edge loop on the inside.  This cannot happen if the twist region lies on the inside.  Hence the twist region must lie on the outside of the two segments shown, as in Figure~\ref{fig:typeC}, middle.  Then the orange face must run all the way across the outside of the twist region, meeting no tentacles or non-prime arcs, as shown in that figure.  Draw an arc through the orange face all the way across the twist region.  On the far side, connect the arc across the state circle to the portion of the square $P$ in the green face.  Now continue to follow $P$ to the right.  It crosses the state circle once more, before joining the orange face where we began, as in Figure~\ref{fig:typeC}, right.  Replace segments of $H_A$ by crossings.  We obtain a simple closed curve in the diagram of the link that meets the link exactly twice, with crossings on either side.  This contradicts the hypothesis that the diagram is prime.

\begin{figure}
  \includegraphics{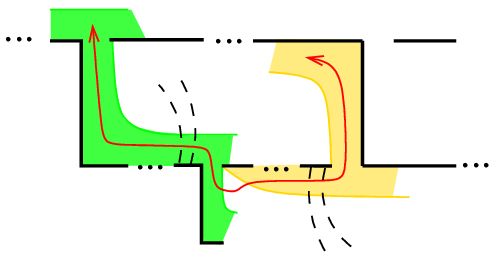}
  \hspace{.1in}
  \includegraphics{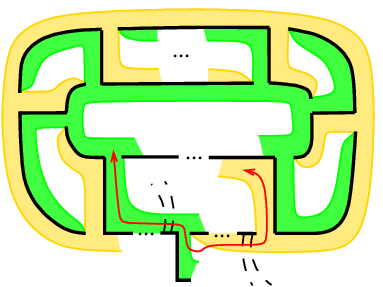}
  \hspace{.1in}
  \includegraphics{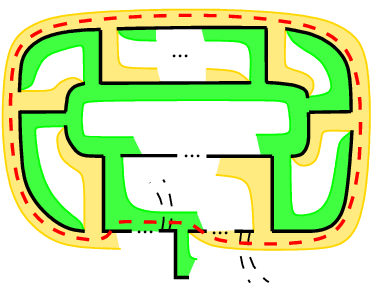}
  \caption{Left: A normal square of type $\mathcal{C}$ runs over a portion of the upper polyhedron of the form shown.  Middle: The precise form  imposed by  2--edge loop condition. Right: Dotted line gives a contradiction to the primeness of the diagram.}
  \label{fig:typeC}
\end{figure}

In the interest of space, we do not reproduce \cite[Figure~6.1]{fkp:gutsjp} for the analysis of the rest of the cases. However, all the cases are handled by a variation of the above argument.

\smallskip
{\underline{Type $\mathcal{E}$.} }  This can also be ruled out by an appeal to the primeness of the diagram.  Because the 2--edge loop in type $\mathcal{E}$ comes from a twist region, the orange face shown, adjacent to the state circle at the bottom, must continue across the bottom of that state circle, meeting no segments or non-prime arcs, until it lies directly opposite the green face running next to the second sement of the 2--edge loop.  Note that the orange face is also directly opposite the green face on the left side of the green non-prime arc.  Now we may draw an arc through the orange face, from the point where it is opposite the non-prime arc to the point where it is opposite the segment of the 2--edge loop.  Connect the arc across the state circle to the portion of $P$ lying in the green face.  This gives a simple closed curve meeting $H_A$ exactly twice on state circles, which in turn gives a loop in the diagram meeting the diagram exactly twice, again contradicting  the hypothesis that the diagram is prime. 

\smallskip 
{\underline{Type $\mathcal{G}$.} } The segments forming the 2--edge loop shown in that figure must belong to a twist region by hypothesis, meaning there can be only segments and state circles of the short resolution of a twist region between those two segments on one side.  But on one side of the two segments there is a segment on the opposite side of the state circle, and on the other side there is a non-prime arc.  Hence type $\mathcal{G}$ cannot occur.

\smallskip
{\underline{Types $\mathcal{D}$ and $\mathcal{F}$.} }  Because the 2--edge loop shown belongs to a twist region, one green tentacle must terminate in a bigon, with no additional green tentacles or non-prime arcs connected to it.  For type $\mathcal{D}$, this green tentacle lies to the right in \cite[Figure~6.1]{fkp:gutsjp}.  For type $\mathcal{F}$, it lies ot the left.  In either case, since $P$ runs through that tentacle, it must meet its second ideal vertex at the end of that green tentacle shown.  In other words, at the end of that green tentacle it must cut off a single vertex in a white face, which defines a triangle.  However, note that the triangle will have the opposite orientation from Convention~\ref{conv:coloring}, contradicting our assumption.

\smallskip
It follows that our normal square $P$ is of type $\mathcal{A}$ or $\mathcal{B}$.  To obtain Figure~\ref{fig:EPD}, we make the following observations.  First, because the 2--edge loop shown in either figure belongs to a twist region, the two segments shown must bound a string of bigons, i.e.\ segments and state circles from the short resolution of a twist region.  In particular, there can be no segments on the opposite side of the state circles shown between these two segments. This implies that some face runs straight across the top of the twist region, meeting no additional segments or state circles, and the orange face runs straight across the bottom of the twist region, meeting no additional segments or state circles.  

For type $\mathcal{A}$, the green face must run across the top of the twist region.  If the boundary of $P$ runs into the green face at the top, then it must run straight across the top, and straight across the bottom through the orange face, and hence $P$ is as shown on the left of Figure~\ref{fig:EPD}.  It may be the case for type $\mathcal{A}$ that the boundary of $P$ runs down a green tentacle, without running across the top.  Then in this case, a portion of $P$ in the orange face runs between these segments, so the two segments form a bigon, and the EPD must meet the ideal vertices of the bigon.  Given our
choice of orientation, the normal square $P$ will have the form of Figure~\ref{fig:EPD}, right.

For type $\mathcal{B}$, we need to show that the green face runs across the top of the twist region.  This follows because the portion of $P$ adjacent to the orange segment on the right is running downstream in a tentacle which forms a bigon of the twist region.  No other tentacles or non-prime arcs can meet this bigon, or the 2--edge loop would not belong to a twist region.  Hence $P$ meets the white face at the end of this tentacle, and joins the green on the opposite side.  It follows that the green face must run across the twist region, with $P$ running across as well to meet in this vertex.  Thus
the normal square $P$ is as shown on the right of Figure~\ref{fig:EPD}.
\end{proof}

The square on the left of Figure \ref{fig:EPD} has a particularly
relevant form.  

\begin{define}\label{def:square-in-twist} 
A diagrammatically compressible square $Q$ with sides $Q_1$, $Q_2$ in
shaded faces is defined to be a \emph{square in a twist region $R$} if
it is of the form shown on the left of Figure \ref{fig:EPD}. More
precisely, there are state circles $C$ and $C'$ and segments $s_1$,
$s_2$ of $H_A$ between $C$ and $C'$, determining $R$, such that:
\begin{enumerate}[(i)]
\item $Q_1$ and $Q_2$ run on adjacent shaded faces near $C'$;
\item $Q_1$ and $Q_2$ run on adjacent shaded faces near $C$;
\item one side of $Q$, say $Q_1$, runs on the top of the twist region
  along $C'$ and then along $s_1$, while $Q_2$ runs along $s_2$ and
  then on the bottom of $R$ along $C$.
\end{enumerate}
\end{define}

\begin{lemma}\label{lemma:EPD-typeA'}
Suppose a prime, $A$--adequate diagram is such that $\GA$ satisfies
the 2--edge loop condition.  Suppose $T \subset S^3 \setminus K$ is an
essential torus, $E \subset T \cap M_A$ is a diagrammatically
compressible annulus, and $Q \subset E$ is a normal square that
parabolically compresses to an EPD in the upper polyhedron.  Then we
may perturb $E$ by a normal isotopy that extends to $T$, so that $Q$
becomes a square in a twist region.
% Definition \ref{def:square-in-twist}.
\end{lemma}

\begin{proof}
Suppose $Q$ parabolically compresses to an EPD.  Then, by Lemma
\ref{lemma:EPD-typeB}, there is a related normal square $P$ that has
one of the two forms of Figure \ref{fig:EPD}.  For each of the forms,
$P$ may either agree with $Q$ or be obtained from $Q$ by moving one
white side across an ideal vertex of the upper polyhedron.

Suppose $P$ is as on the left of Figure \ref{fig:EPD}.  If $P$ and $Q$
agree, then we are finished.  Otherwise, $P$ differs from $Q$ in that
one of its white sides is on the opposite side of an ideal vertex from
$Q$.  We may assume that this white side is located to the right of
the figure shown.  If the ideal vertex in question does not contain
the right--most segment of the twist region of Figure \ref{fig:EPD},
left, then adjusting $P$ to form $Q$ has no effect on the region
shown, and the diagram is as claimed.  If the ideal vertex does
contain the right--most segment, then $Q$ must have a white side on
the opposite side of the corresponding ideal vertex, which lies in the
right--most bigon of the twist region shown, and $Q$ now has the form
of a square on the right of Figure \ref{fig:EPD}.  

Now suppose that $Q$ is as shown on the right of Figure \ref{fig:EPD},
and either $P$ agrees with $Q$, or $P$ is as on the left of Figure
\ref{fig:EPD}, as in the previous paragraph.  Then the white side of
the square on the right of the figure lies in a bigon white face.
This white face is glued to a bigon in a lower polyhedron, and the
annulus runs through the bigon.  Isotope the annulus through the
bigon, isotoping $Q$ to have a white side cutting off a vertex on the
opposite side of the bigon.  This pulls the shaded sides of $Q$ along as
well, pulling the side of $Q$ in the right--most orange tentacle to
only meet the head of that tentacle, and pulling the side in the green
to run downstream adjacent to a green tentacle.  The isotopy can be
performed in a neighborhood of the white face in $M_A$, affecting only
the square $Q$ and the square in the lower polyhedron glued to $Q$ at
this face, but only in a neighborhood of this single white face.
Moreover, the isotopy may be extended into a small neighborhood of
$\widetilde{S_A}$, to extend to all of $T$.  After this isotopy, $Q$
has the form claimed in the lemma.

The only remaining case is that the square $P$ has the form on the
right of Figure \ref{fig:EPD}, but it differs from the original square
$Q$ in that we slid a white side of $Q$ through an ideal vertex.
Because we were able to choose the orientation on one vertex, we may
assume that the ideal vertex in question is at the right of that
figure.  Notice that the side of $P$ that must be adjusted lies in a
bigon in a twist region, with state circle $C'$ on top, and $C$ on
bottom of the twist region.  The ideal vertex cut off by $P$ is a
portion of the graph $H_A$ consisting of a small arc on $C'$, a
segment $s$ connecting $C'$ and $C$, a portion of $C$, and possibly
more segments.  Hence when we slide over this ideal vertex, the
portion of $P$ running through the orange tentacle adjacent $s$ slides
out of this tentacle, and a portion of the square in the green face
will be pulled through the green tentacle adjacent to $s$.  When
finished, the result will be as claimed in the statement of the lemma.
\end{proof}

The next lemma will allow us to show that, assuming the 2--edge loop condition, 
no diagrammatically
incompressible annuli in $M_A$ come from essential tori.  This is done
in Lemmas \ref{lemma:upper-idealvertex} and
\ref{lemma:lower-idealvertex}.

\begin{lemma}\label{lemma:square-adjvertex}
Let $D(K)$ be an $A$--adequate diagram, with all--$A$ state surface
$S_A$, and let $T$ be an incompressible torus.  Suppose that a normal
square $Q \subset T \cap M_A$ lies in the upper polyhedron, with sides
$Q_1$ and $Q_2$ in shaded faces, which satisfies the following
property.  There exists a point $p$ on a state circle and arcs
$q_1\subset Q_1$ and $q_2\subset Q_2$ that can be isotoped to lie on
either side of an $\epsilon$ neighborhood of $p$, while maintaining
the condition that $Q_1$ and $Q_2$ are simple with respect to their
shaded faces (Remark \ref{remark:simple}).  

Then
$T\setminus\widetilde{S_A}$ consists of exactly two annuli, one in
$M_A$ and one in the $I$--bundle $N(S_A)$.  Moreover, subarcs of $q_1$
and $q_2$ are glued together in $T$ when we glue opposite sides of
$N(S_A)$, to recover $S^3\setminus K$.
\end{lemma}

\begin{proof}
The point $p$ on a state circle corresponds to a point, which we will
also call $p$, on the knot $K$.  Let $N_p$ be a tubular neighborhood
of $p$ in $S^3$.  We may take this neighborhood small enough that
$N_p\cap M_A$ lies entirely in the upper polyhedron, and $N_p\cap
N(S_A)$ is a trivial $I$--bundle, of the form $D\times [-1,1]$ for
some disk $D$ in $S_A$.  Moreover, we may isotope $q_1$ and $q_2$ (and
all of $T$), if necessary, so that $q_1$ and $q_2$ each run through
$N_p$.  Since $Q$ has opposite sides containing $q_1$ and $q_2$, and
since $Q$ lies in a ball (the upper polyhedron), we may isotope $Q$
relative to its boundary so that a sub-rectangle $Q'$ of $Q$ has opposite
sides on $q_1$ and $q_2$, and lies completely in $N_p$.

Now consider $Q'$.  This is a rectangle in the tube $N_p$ surrounding
the point $p$ on the knot $K$.  It has one boundary component, say
$q_1$, on $D\times \{-1\}$, and one boundary component, $q_2$, on
$D\times\{1\}$.  Notice it almost encircles $K$ to form an annulus
whose core is a single meridian, except that it is cut by $N(S_A)$.
Inside $N(S_A)$, both $q_1$ and $q_2$ are glued to annuli which are
isotopic to vertical annuli in the $I$--bundle.  The proof will be
complete when we show that $q_1$ and $q_2$ are glued to the same
vertical annulus in the $I$--bundle.

So suppose not.  Suppose $q_1$ is glued to the annulus $E_1$ and $q_2$
is glued to the annulus $E_2$.  Since $E_1$ and $E_2$ are both subsets
of $T$, they are disjoint.  They are also compact, connected, and they
lie inside the compact set $N(S_A)\setminus N_\delta(K)$, where
$N_\delta(K)$ is a sufficiently small tubular neighborhood of $K$.
Consider the projection $\pi \co N(S_A)\setminus N_\delta(K) \to S_A$.
This is a continuous map on a compact set, hence it is proper.  So
$\pi(E_1)$ and $\pi(E_2)$ are compact in $S_A$.  Since $E_i$ is the
image of a vertical annulus, it lies in a bounded neighborhood of a
core curve $\tau_i$ on $S_A$, $i=1, 2$.

Now, let $\sigma$ be a simple arc in $S_A$ with one endpoint on $p$,
exiting $N_p \cap S_A$ by crossing through $\pi(E_1)$ and $\pi(E_2)$,
and such that $\sigma$ meets each of $\tau_1$, $\tau_2$ transversely
exactly once, and the final endpoint of $\sigma$ is disjoint from
$\pi(E_1)$ and $\pi(E_2)$.  Note that if we restrict the $I$--bundle
$N(S_A)$ to $\sigma$, we obtain a trivial $I$--bundle over a line,
which is a rectangle $R \cong \sigma\times[-1,1]$.  Note also that $R$
intersects both $E_1$ and $E_2$ in $R\cap N_p$.  In particular, an arc
of intersection of $E_1 \cap R$ has endpoint in $N_p$ at a point on
$\sigma \times \{-1\}$, and an arc of intersection of $E_2 \cap R$ has
an endpoint in $N_p$ on $\sigma\times \{1\}$.

Consider the other endpoint of $E_1 \cap R$.  By choice of $\sigma$,
this endpoint must be either on $\sigma \times \{-1\}$ or on
$\sigma\times\{1\}$.  If $E_1\cap R$ has both endpoints on
$\sigma\times\{-1\}$, then there is a disk in $R$ with boundary on
$E_1\cap R$ and on $\sigma \times\{-1\}$.  We may isotope $E_1$
through this disk, removing the intersection of $E_1$ with $R$, hence
pushing the arc $q_1$ away from $p$.  Such an isotopy is impossible
under the assumption that the sides of $Q$ were simple.  Hence
$E_1\cap R$ has one endpoint on $\sigma \times \{-1\}$ and one on
$\sigma\times\{1\}$.  Similarly, $E_2\cap R$ has one endpoint on
$\sigma\times\{1\}$ and one on $\sigma\times\{-1\}$.

Because $E_1$ and $E_2$ are embedded, these endpoints cannot
interleave.  Thus either both endpoints of $E_1\cap R$ will lie in
$N_p \cap R$ or both endpoints of $E_2\cap R$ will lie in $N_p\cap
R$.  Say both endpoints of $E_2$ lie in $N_p\cap R$.

Now, $E_2$ must connect to an annulus in $M_A$ on both of its boundary
components.  One boundary component connects to an annulus containing
$Q$.  The other cannot connect to $Q$ because its initial endpoint on
$R$ lies interior to $q_1 \cap R$.  So it connects to some new
rectangle $Q''$.  But consider the shaded sides of $Q''$.  One lies
parallel to $q_1$, but interior to $q_1$ (i.e.\ closer to $K$) on
$\widetilde{S_A}$.  The other cannot also be parallel to $q_1$ in
$N_p$, or we could eliminate an intersection of $T$ with $S_A$.  But
$Q''$ cannot intersect $Q$, hence $Q''$ must lie parallel to $Q'$ in
$N_p$, and its opposite boundary component is parallel to $q_2$ on
$\widetilde{S_A}$, but closer to $K$.

Now we may repeat the entire above argument with $Q''$, to obtain a
square interior to $Q''$ in $N_p$.  Since each square we pick up at
each step is interior to all previous squares, we obtain an infinite
sequence of squares, each lying on $T$.  This is a contradiction.

Thus $q_1$ and $q_2$ are connected by a single annulus $E \subset
N(S_A)$, which can be isotoped to be vertical.  
\end{proof}

%If $\GA$ satisfies the 2--edge loop condition, and $Q$ is a square in
%the upper polyhedron that does not parabolically compress to a single
%ideal vertex, then Lemma \ref{lemma:EPD-typeA'} implies that the
%hypothesis of Lemma \ref{lemma:square-adjvertex} is satisfied for $Q$.
%This fact will be important in the remainder of this section.

\begin{lemma}\label{lemma:upper-idealvertex}
Suppose a prime, $A$--adequate is such that $\GA$ satisfies the
2--edge loop condition.  Then each diagrammatically compressible
square in the upper polyhedron coming from an essential torus
must bound a single ideal vertex.
\end{lemma}

\begin{proof}
Suppose we have a square $Q \subset T \cap M_A$ that does not encircle
a single ideal vertex. Then $Q$ parabolically compresses to an EPD.
By Lemma \ref{lemma:EPD-typeA'}, we may isotope the torus $T$
containing $Q$ so that $Q$ is a square in a twist region. By parts
(i)-(ii) of Definition \ref{def:square-in-twist}, there are points
$p\in C$ and $p'\in C'$ for which the hypotheses of Lemma
\ref{lemma:square-adjvertex} are satisfied. By the conclusion of that
lemma, there is only one annulus of $T\cap M_A$, and the shaded sides
of $Q$ that run near $p'$ (resp. $p$) are glued to each other in $T$.

However, note that between these two pairs of glued arcs, the sides of
the square run adjacent to distinct segments of $H_A$: $Q_1$ runs
along $s_1$, while $Q_2$ runs along $s_2$.  By the above paragraph,
these sides of $Q$ are glued to each other when we collapse $N(S_A)$
to $S_A$ to recover the torus $T$.  Hence the shaded sides of $Q$ on
$\widetilde{S_A}$ have homotopic projection to $S_A$.  The graph $\GA$
is a spine for $S_A$, so we may homotope the arcs in $S_A$ to run over
the same edges of $\GA$ in the same order.  But because the arcs run
adjacent to distinct segments of $H_A$, they run over distinct edges
(corresponding to these segments) in $\GA$.  This is a contradiction.
\end{proof}

\begin{lemma}\label{lemma:lower-idealvertex}
Suppose $D(K)$ is is a prime, $A$--adequate diagram such that $\GA$
satisfies the 2--edge loop condition.  Suppose an essential torus $T$
contains a diagrammatically compressible annulus $E \subset T \cap
M_A$.  Then $E$ is the only component of $T \cap M_A$, and every
normal square comprising $E$ encircles a single ideal vertex in its
ambient polyhedron.
\end{lemma}

\begin{proof}
By Lemma \ref{lemma:upper-idealvertex}, each diagrammatically
compressible square in the upper polyhedron bounds a single ideal
vertex.  It follows from Lemma \ref{lemma:square-adjvertex} that each
such square has one side in a shaded face glued to the other side when 
we re-glue $N(S_A)$.  Thus $E$ is the only annulus of $T \cap M_A$.

Each square $Q$ in the upper polyhedron is attached to a
diagrammatically compressible square $Q'$ in the lower polyhedron.  A
diagrammatically compressible square in the lower polyhedron
intersects the diagram graph exactly four times, adjacent to
crossings.  Because the diagram is $A$--twist reduced (Lemma
\ref{lemma:twist-reduced}), the square must encircle either an ideal
vertex or a string of bigons.

Suppose the square $Q'$ encircles a nonempty string of bigons.  Because
each side in a white face is glued to a square in the upper
polyhedron, and each square in the upper polyhedron has shaded sides
glued when we reattach sides of $N(S_A)$ to form $S^3\setminus K$, the
square in the lower polyhedron must also have its sides in shaded
faces glued when we form $S^3\setminus K$.

But now consider the way these sides run through the graph $\GA$,
which forms a spine for the surface $S_A$.  One of the curves runs
through a vertex associated with a state circle on one side of the
twist region, and the other runs through a distinct vertex on the
other side of the twist region.  Because the sides in white faces are
adjacent to distinct crossings, the curves cannot be homotopic in
$S_A$, hence they cannot be glued in the square.  This contradiction
implies that the square $Q'$ runs over a single ideal vertex.  
\end{proof}
%%%%%%%%%%%%%%%%%%%%%%%%%%%%%%%%%%%%%%%%%%%%%%%%%%%%%%%%%%%%%%%%%
%%%%%%%%%%%%%%%%%%%%%%%%%%%%%%%%%%%%%%%%%%%%%%%%%%%%%%%%%%%%%%%%%

We are now ready to show that semi-adequate links that satisfy the
hypothesis of Theorem \ref{thm:main} have atoroidal complements.

\begin{theorem}\label{thm:tori}
If $D(K)$ is a prime, $A$--adequate diagram such that $\GA$ satisfies
the 2--edge loop condition, then $S^3\setminus K$ contains no embedded
essential tori.
\end{theorem}

\begin{proof}
Suppose $T$ is an embedded essential torus in $S^3\setminus K$.  By
Lemma \ref{lemma:tori-cutto-annuli}, we may take $T$ to be embedded in
such a way that $\widetilde{S_A}$ cuts it into an even number of
essential annuli, half embedded in $M_A$ and half embedded in the
$I$--bundle $N(S_A)$.  By Lemma \ref{lemma:incomprannuli}, $T \cap
M_A$ must contain a diagrammatically compressible annulus $E$.

But by Lemma \ref{lemma:lower-idealvertex}, if the essential torus
gives rise to a diagrammatically compressible annulus $E$, then $E$ is
the only component of $T \cap M_A$, and every normal square of $E$
encircles a single ideal vertex in its polyhedron.  In that case,
when we glue opposite sides of $N(S_A)$ to recover $S^3\setminus K$,
the sides of $E$ will be identified to encircle knot strands, and it
follows that the torus $T$ is actually boundary parallel, and not
essential.
\end{proof}

%%%%%%%%%%%%%%%%%%%%%%%%%%%%%%%%%%%%%%%%%%%%%%%%%%%%%%%%%%%%%%%%%
%%%%%%%%%%%%%%%%%%%%%%%%%%%%%%%%%%%%%%%%%%%%%%%%%%%%%%%%%%%%%%%%%

\section{Seifert fibered link complements}\label{sec:seifert}

Our work in the previous sections reduces the proof of the main result
to the case of atoroidal link complements. As we remarked in the introduction, the diagram $D(K)$ is assumed connected, which implies that $K$ is non-split and $S^3 \setminus K$ is irreducible. By work of Thurston \cite{thurston:bulletin}, an irreducible,
atoroidal $3$--manifold is either hyperbolic or Seifert fibered.  
To finish the proof of Theorem \ref{thm:main}  we need to treat the case of
Seifert fibered link complements. This is done in the following theorem.

\begin{theorem}\label{thm:SFL}
Let $D(K)$ be a prime, connected, $A$--adequate diagram such that
$\GA$ satisfies the 2--edge loop condition.  If $S^3 \setminus K$ is
Seifert fibered, then $D(K)$ is the standard diagram of a $(2,q)$
torus link.
\end{theorem}

\begin{proof}
Let $\GRA$ denote the reduced state graph of $D(K)$, obtained from
$\GA$ by removing all the duplicate edges.  Recall that the
\emph{guts} of $S^3 \setminus K$ relative to the surface $S_A$,
denoted $\guts(S^3 \setminus K, \, S_A)$, is the complement of the
maximal $I$--bundle in $M_A = S^3 \cut S_A$. In \cite[Corollary
  5.19]{fkp:gutsjp} we proved that, when all $2$--edge loops in $\GA$
belong to twist regions,
$$ -\chi (\guts(S^3 \setminus K, \, S_A) ) \: = \: \max \{ -\chi(\GRA), 0 \}.$$

Furthermore, the work of Agol \cite{agol:guts}, as generalized by
Kuessner \cite{kuessner:guts}, says that guts can be used to estimate
the Gromov norm of $S^3 \setminus K$:
$$ ||S^3\setminus K|| \geq - 2 \, \chi (\guts(S^3 \setminus K, \, S_A) ) \: \geq \:  -2 \, \chi (\GRA).$$

Recall that the Gromov norm $|| M ||$ of a $3$--manifold $M$ is
positive whenever the JSJ decomposition has one or more hyperbolic
pieces \cite{gromov, benepetronio}. In particular, if $S^3 \setminus
K$ is Seifert fibered, we have $||S^3\setminus K||=0$, hence
$\chi(\GRA) \geq 0$.

Next, recall that the graph $\GA$ can be given the structure of a
\emph{ribbon graph} \cite{dasbach-futer...}, and as such it can be
embedded on a standard closed orientable surface (called the
\emph{Tureav surface} of $D(K)$) so that it defines a cellulation
\cite{dasbach-futer..., turaevs}. The genus of this surface is called
the \emph {Turaev genus} of $D(K)$.  The Turaev genus $g(D)$ satisfies
$$ 2g(D) \: = \: 2-v(\GA)+ e(\GA)-f(\GA) \: = \: 2-\chi(\GRA)+ (e(\GA)-e(\GRA))-f(\GA),
$$ 
where $v(\GA), e(\GA), f(\GA)$ denote the number of vertices, edges
and faces, respectively, of the aforementioned cellulation, and
$e(\GRA)$ is the number of edges of $\GRA$.

Since 2--edge loops in $\GA$ belong to twist regions of $D(K)$, for
every edge in $e(\GA)-e(\GRA))$ there is a bigon face in $f(\GA)$ that
cancels that edge. Furthermore, if $D(K)$ has more than one twist
region --- if it is not the standard diagram of a $(2,q)$ torus link ---
there must also be at least one non-bigon face. Therefore,
$$(e(\GA)-e(\GRA))-f(\GA) \: \leq \: -1.$$
Furthermore, we have seen above that $\chi(\GRA) \geq 0$, hence
$$2g(D) \: = \: 2-\chi(\GRA)+ (e(\GA)-e(\GRA)-f(\GA)) \: \leq \:  2 - 0 -1.$$
Since $g(D)$ is a non-negative integer, we conclude that $g(D)=0$.

This in turn, implies that the diagram $D$ is alternating; see
Corollary 4.6 of \cite{dasbach-futer...}.  Thus $D$ a prime,
alternating diagram that represents a Seifert fibered link.  Now the
work of Menasco \cite{menasco:incompress} implies that $D$ is the standard diagram of
a $(2,q)$ torus link.
\end{proof}

Theorem \ref{thm:main} follows immediately by combining Theorem
\ref{thm:tori}, Theorem \ref{thm:SFL}, and Thurston's hyperbolization
theorem for link complements \cite{thurston:bulletin}.
Now we finish the proofs of Corollaries \ref{braids} and \ref{cor:prime}.

\begin{proof}[Proof of Corollary \ref{braids}]
Without loss of generality, assume $r_j \geq 3$ for all $j$. Then  the diagram $D_b$ is a
prime, $A$--adequate diagram and  the corresponding state graph $\GA$
contains no 2--edge loops at all.  Thus Theorem \ref{thm:main} implies
that $K$ is hyperbolic.
\end{proof}

\begin{proof}[Proof of Corollary \ref{cor:prime}]
In  \cite[Corollary 3.21]{fkp:gutsjp}, we show that for a
non-split, prime link $K$, any semi-adequate diagram $D(K)$
without nugatory crossings must be prime.

Conversely, if $D(K)$ is prime and semi-adequate, then Theorem
\ref{thm:main} implies $K$ is hyperbolic or a $(2,q)$ torus link.  Hence the link must also be
prime.
\end{proof}

\bibliographystyle{hamsplain} \bibliography{biblio}

\providecommand{\bysame}{\leavevmode\hbox to3em{\hrulefill}\thinspace}
\providecommand{\href}[2]{#2}
\begin{thebibliography}{10}

\bibitem{adams:augmented}
Colin~C. Adams, \emph{Augmented alternating link complements are hyperbolic},
  Low-dimensional topology and {K}leinian groups ({C}oventry/{D}urham, 1984),
  London Math. Soc. Lecture Note Ser., vol. 112, Cambridge Univ. Press,
  Cambridge, 1986, pp.~115--130.

\bibitem{Adams-toroidally}
\bysame, \emph{Toroidally alternating knots and links}, Topology \textbf{33}
  (1994), no.~2, 353--369.

\bibitem{Adams-survey}
\bysame, \emph{Hyperbolic knots}, Handbook of knot theory, Elsevier B. V.,
  Amsterdam, 2005, pp.~1--18.

\bibitem{agol:guts}
Ian Agol, \emph{Lower bounds on volumes of hyperbolic {H}aken 3-manifolds},
  \mbox{arXiv:math/9906182}.

\bibitem{benepetronio}
Riccardo Benedetti and Carlo Petronio, \emph{Branched standard spines of
  {$3$}-manifolds}, Lecture Notes in Mathematics, vol. 1653, Springer-Verlag,
  Berlin, 1997.

\bibitem{bonsieb:monograph}
Francis Bonahon and Laurent Siebenmann, \emph{{New Geometric Splittings of
  Classical Knots, and the Classification and Symmetries of Arborescent
  Knots}}, Geometry \& Topology Monographs, to appear, {\tt http://\allowbreak
  www-bcf.usc.edu/\allowbreak \~{ }fbonahon/\allowbreak Research/\allowbreak
  Preprints/\allowbreak Preprints.html}.

\bibitem{cromwell-book}
Peter~R. Cromwell, \emph{Knots and links}, Cambridge University Press,
  Cambridge, 2004.

\bibitem{dasbach-futer...}
Oliver~T. Dasbach, David Futer, Efstratia Kalfagianni, Xiao-Song Lin, and
  Neal~W. Stoltzfus, \emph{The {J}ones polynomial and graphs on surfaces},
  Journal of Combinatorial Theory Ser. B \textbf{98} (2008), no.~2, 384--399.

\bibitem{fg:arborescent}
David Futer and Fran{\c{c}}ois Gu{\'e}ritaud, \emph{Angled decompositions of
  arborescent link complements}, Proc. Lond. Math. Soc. (3) \textbf{98} (2009),
  no.~2, 325--364.

\bibitem{fkp:filling}
David Futer, Efstratia Kalfagianni, and Jessica~S. Purcell, \emph{{Dehn
  filling, volume, and the Jones polynomial}}, J. Differential Geom.
  \textbf{78} (2008), no.~3, 429--464.

\bibitem{fkp:conway}
\bysame, \emph{Symmetric links and {C}onway sums: volume and {J}ones
  polynomial}, Math. Res. Lett. \textbf{16} (2009), no.~2, 233--253.

\bibitem{fkp:farey}
\bysame, \emph{{Cusp areas of {F}arey manifolds and applications to knot
  theory}}, Int. Math. Res. Not. IMRN \textbf{2010} (2010), no.~23, 4434--4497.

\bibitem{fkp:gutsjp}
\bysame, \emph{Guts of surfaces and the colored {J}ones polynomial}, Lecture
  Notes in Mathematics, vol. 2069, Springer, Heidelberg, 2013.

\bibitem{fkp:survey}
\bysame, \emph{Jones polynomials, volume, and essential knot surfaces: a
  survey}, Knots in Poland III, Part I, Banach Center Publ., vol. 100, Polish
  Acad. Sci. Inst. Math., Warsaw, 2014, \mbox{arXiv:1110.6388}, pp.~51--77.

\bibitem{fkp:qsf}
\bysame, \emph{Quasifuchsian state surfaces}, Trans. Amer. Math. Soc.
  \textbf{366} (2014), no.~8, 4323--4343.

\bibitem{fp:twisted}
David Futer and Jessica~S. Purcell, \emph{Links with no exceptional surgeries},
  Comment. Math. Helv. \textbf{82} (2007), no.~3, 629--664.

\bibitem{giambrone}
Adam Giambrone, \emph{Combinatorics of link diagrams and volume},
  arXiv:1310.8414.

\bibitem{gromov}
Michael Gromov, \emph{Volume and bounded cohomology}, Inst. Hautes \'Etudes
  Sci. Publ. Math. (1982), no.~56, 5--99 (1983).

\bibitem{hatcher:3manifolds}
Allen Hatcher, \emph{Notes on basic 3-manifold topology}, {\tt
  http://www.math.cornell.edu/\allowbreak \~{ }hatcher/3M/\allowbreak
  3Mdownloads.html}.

\bibitem{KaufJones}
Louis~H. Kauffman, \emph{State models and the {J}ones polynomial}, Topology
  \textbf{26} (1987), no.~3, 395--407.

\bibitem{kuessner:guts}
Thilo Kuessner, \emph{Guts of surfaces in punctured-torus bundles}, Arch. Math.
  (Basel) \textbf{86} (2006), no.~2, 176--184.

\bibitem{lackenby:volume-alt}
Marc Lackenby, \emph{The volume of hyperbolic alternating link complements},
  Proc. London Math. Soc. (3) \textbf{88} (2004), no.~1, 204--224, With an
  appendix by Ian Agol and Dylan Thurston.

\bibitem{lick-thistle}
W.~B.~Raymond Lickorish and Morwen~B. Thistlethwaite, \emph{Some links with
  nontrivial polynomials and their crossing-numbers}, Comment. Math. Helv.
  \textbf{63} (1988), no.~4, 527--539.

\bibitem{menasco:incompress}
William~W. Menasco, \emph{Closed incompressible surfaces in alternating knot
  and link complements}, Topology \textbf{23} (1984), no.~1, 37--44.

\bibitem{ozawa:positive}
Makoto Ozawa, \emph{Closed incompressible surfaces in the complements of
  positive knots}, Comment. Math. Helv. \textbf{77} (2002), no.~2, 235--243.

\bibitem{ozawa}
\bysame, \emph{Essential state surfaces for knots and links}, J. Aust. Math.
  Soc. \textbf{91} (2011), no.~3, 391--404.

\bibitem{purcell2}
Jessica~S. Purcell, \emph{Hyperbolic geometry of multiply twisted knots}, Comm.
  Anal. Geom. \textbf{18} (2010), no.~1, 101--120.

\bibitem{purcell1}
\bysame, \emph{On multiply twisted knots that are {S}eifert fibered or
  toroidal}, Comm. Anal. Geom. \textbf{18} (2010), no.~2, 219--256.

\bibitem{thi:adequate}
Morwen~B. Thistlethwaite, \emph{On the {K}auffman polynomial of an adequate
  link}, Invent. Math. \textbf{93} (1988), no.~2, 285--296.

\bibitem{thurston:bulletin}
William~P. Thurston, \emph{Three-dimensional manifolds, {K}leinian groups and
  hyperbolic geometry}, Bull. Amer. Math. Soc. (N.S.) \textbf{6} (1982), no.~3,
  357--381.

\bibitem{turaevs}
Vladimir~G. Turaev, \emph{A simple proof of the {M}urasugi and {K}auffman
  theorems on alternating links}, Enseign. Math. (2) \textbf{33} (1987),
  no.~3-4, 203--225.

\end{thebibliography}

\end{document}